\documentclass[11pt]{amsart}
\usepackage{fullpage}

\usepackage{hyperref}
\usepackage{amsmath,amsfonts,amssymb}
\usepackage{mathrsfs}
\usepackage{verbatim}
\usepackage{amsthm}
\usepackage{mathrsfs} 
\usepackage{color}

\newtheorem{theorem}{Theorem}[section]
 \newtheorem{corollary}[theorem]{Corollary}
 \newtheorem{lemma}[theorem]{Lemma}
 \newtheorem{proposition}[theorem]{Proposition}
 \theoremstyle{definition}

 \theoremstyle{remark}
 \newtheorem{remark}[theorem]{Remark}
  
 \numberwithin{equation}{section}

\def \bC {\mathbb C}

\def \bH {\mathbb H}

\def \bN {\mathbb N}

\def \bR {\mathbb R}

\def \bT {\mathbb T}

\def \bZ {\mathbb Z}

\def \cD {\mathcal D}

\def \cL {\mathcal L}

\def \cR {\mathcal R}
\def \cS {\mathcal S}

\def \fg {\mathfrak g}

\def \fU {\mathfrak U}

\def \sL{\mathscr L}

\def \tr {{\rm Tr}}

\def \id {\text{\rm I}}

\def \eps {\varepsilon}
\def \vol {{\rm vol}}

\begin{document}
\author[V. Fischer]{V\'eronique Fischer}
\address[V. Fischer]%
{University of Bath, Department of Mathematical Sciences, Bath, BA2 7AY, UK} 
\email{v.c.m.fischer@bath.ac.uk}

\title{Asymptotics and zeta functions on compact nilmanifolds}

\medskip




\begin{abstract}
In this paper, we obtain asymptotic formulae  on nilmanifolds $\Gamma \backslash G$, 
where 
 $G$ is any stratified (or even graded) nilpotent Lie group equipped with a co-compact discrete subgroup $\Gamma$. 
 We study especially the asymptotics related to the sub-Laplacians naturally coming from the stratified structure of the group $G$
(and more generally any positive Rockland operators when $G$ is graded).
We show that the short-time asymptotic on the diagonal of the kernels of spectral multipliers contains only a single non-trivial term. 
We also study the associated  zeta functions. 
\end{abstract}

\maketitle

\makeatletter
\renewcommand\l@subsection{\@tocline{2}{0pt}{3pc}{5pc}{}}
\makeatother

\tableofcontents

\section{Introduction}

For more than a century, the global theory of elliptic operators such as Riemannian Laplacians and Dirac operators has attracted interests from many branches of mathematics, especially spectral analysis on manifolds, but also in theoretical physics. 
From the  1970's \cite{folland+stein_74,roth+stein}, hypoelliptic operators and in particular sub-Laplacians have been intensively studied as well, 
with a particular emphasis on analysis on sub-Riemannian manifolds, predominantly CR and  contact manifolds (see e.g.  \cite{Beals+Greiner,PongeAMS2008} and the references therein). 

In this paper, we determine the complete and explicit short-time  expansion for the spectral multipliers of a  large class of hypoelliptic operators, together with properties of their spectral zeta functions. The hypoelliptic operators we consider include the intrinsic sub-Laplacians 
on any compact stratified nilmanifold $M$, 
i.e. the quotient  $M=\Gamma \backslash G$ of a 
stratified nilpotent Lie group (also called Carnot groups  in sub-Riemannian geometry) by a discrete co-compact subgroup $\Gamma$. By intrinsic, we mean that the sub-Laplacians  we consider on $M$ are the sum of squares of the left-invariant vector fields associated with a choice of basis of the first stratum  for $G$. Previous results on the specific setting of nilmanifolds include  the case of the Heisenberg groups and more generally for  step-two nilpotent Lie groups (see e.g. \cite{Bauer+,Bauer+21} and references therein). 
Nilmanifolds  can be seen in harmonic analysis on nilpotent Lie groups as analogues of `flat and commutative' tori  
and from a geometric viewpoint as a tower of iterated torus bundles, leading to interesting descriptions for instance in complex geometry for Heisenberg nilmanifolds 
  \cite{Folland2004}.

Our results hold 
 in an even more general setting: for any positive Rockland operator $\cR$ on any graded Lie group $G$ equipped with  any discrete co-compact subgroup $\Gamma$.  With graded underlying groups, nilmanifolds  form an interesting class of filtered manifolds: they can give examples of CR and contact manifolds when the underlying group is the Heisenberg group, but they can also produce examples of filtered manifolds outside of the   equiregular sub-Riemannian realm.  
This paper is set in this setting  not for generality's sake but in the light of very recent results on filtered manifolds. Indeed, in the past decade,  considerable progress has been achieved in the study  of spectral properties of  hypoellliptic operators  not only for sub-Laplacians on  sub-Riemannian manifolds but more generally (in the equiregular case) for Rockland operators  on filtered manifolds. This is mainly due  to the development of groupoid techniques (see e.g. \cite{vanErp,choi+ponge,vanErp+Y}). These techniques were used  recently by Dave and Haller  \cite{Dave+Haller}
 to obtain for any Rockland operator on any compact filtered manifold  short-time
heat kernel expansions with consequences regarding  zeta functions  and  index theory. 
From those expansions, together with the local nature of the heat  asymptotics and the homogeneity properties in the group setting, it may be expected that the heat expansion  on nilmanifolds has only one term. 
However, to the author's knowledge, this has not been proved. The main contribution of this paper is to show this property with a clear and simple proof on  nilmanifolds, and moreover for any $\cR$-spectral multiplier - not only for the heat operator $e^{-t\cR}$. We believe that our paper  brings a new understanding of previous results such as in the works of Bauer, Furutani and  Iwasaki \cite{Bauer+,Bauer+21} in the step-two case, and their links with the recent analysis on filtered manifolds \cite{vanErp+Y,Dave+Haller}

The main results of this paper can be summarised in the following statement:
 
\begin{theorem}
\label{thm_main}
Let $\Gamma$ be a discrete co-compact subgroup of 
a graded nilpotent Lie group $G$. 
The resulting compact nilmanifold is denoted by  $M=\Gamma \backslash G$. 
Let $\cR$ be a positive, differential and homogeneous Rockland operator on $G$.
We denote by $\cR_M$ the corresponding operator on $M$. We keep the same notation for the self-adjoint extensions of $\cR$ on $L^2(G)$ and $\cR_M$ on $L^2(M)$. 

\begin{description}
\item[(i) Short-time asymptotics]
Let $\psi\in \cS(\bR)$.
The convolution kernel $\kappa=\psi(\cR)\delta_0$ of $\psi(\cR)$ is Schwartz on $G$. 
For each $t>0$, 
the operator $\psi(t\cR_M)$ is trace-class and its integral kernel $K_t$ is smooth on $M\times M$. 
These kernels have the following asymptotics on the diagonal as $t\to 0$:
$$
K_t(\dot x,\dot x)= \kappa (0) \, t^{-Q/\nu}+O(t)^\infty, \quad \dot x\in M,
$$
where $Q$ and $\nu$ denote the homogeneous dimension of $G$ and the homogeneous degree of $\cR$ respectively. 
The traces of the operators $\psi(t\cR_M)$ have the following asymptotics as $t\to 0$ 
	$$
\tr \left(\psi(t\cR_M)\right) =  
\vol (M) \kappa(0) \, t^{-Q/\nu}+O(t)^\infty, 
$$
where $\vol (M)$ denotes the volume of $M$.

Furthermore, 
$$
\kappa(0) = c_0\int_0^\infty \psi (\lambda)  \lambda^{\frac Q \nu} \frac{d\lambda}{\lambda},
$$
where $c_0$ is a constant of $G$ and $\cR$ only.
Consequently, $c_0 =p_1(0) / \Gamma(Q/\nu)$ where $p_1(0)= e^{-\cR}\delta_0(0)$ is the heat kernel at time 1 and at $x=0$. 
 \\
 
\item [(ii) Spectral zeta function] 
We denote by $0=\lambda_0<\lambda_1 \leq \lambda_2 \leq \ldots$ 
 the eigenvalues of $\cR_M$ counted with multiplicity, and by 
$\zeta_{\cR_M}(s) =\sum_{j=1}^\infty \lambda_j^{-s}$
the spectral zeta function of $\cR_M$. 

The function 
$\zeta_{\cR_M}$ is holomorphic on $\{\Re s>Q/\nu\}$ and 
 admits a meromorphic extension to $\bC$ with only one pole.
 This pole is simple, located  at $s=Q/\nu$ and with residue  $ c_0 \vol (M)  $. 
 Furthermore, we know the following values:
$$
\zeta_{\cR_M}(0)=-1
\quad\mbox{and}\quad
\zeta_{\cR_M}(s)=0 \ \mbox{for}\ s=-1,-2,\ldots
$$
\end{description}
\end{theorem}

To the author's knowledge,
the values of the zeta function in Theorem \ref{thm_main}  are entirely new for sub-Laplacians in step greater than three, and for Rockland operators on nilmanifolds. 
{It is interesting that this class of zeta functions behaves  like the spectral zeta function of the Laplacian on the flat torus for the properties described in Theorem \ref{thm_main}, i.e. having a simple pole and the `trivial' zeros at negative integers.}

Applying Theorem \ref{thm_main} to $\psi(\lambda)=e^{-\lambda}$, $\lambda\geq 0$, we obtain the following short-time asymptotic for the heat kernels
 $p_t=e^{- t\cR}\delta_0$ and $K_t$  of $\cR$ and $\cR_M$  as $t \to 0$:
 $$
 \forall \dot x\in M \quad 
K_t(\dot x,\dot x)= p_1(0) \, t^{-Q/\nu}+O(t)^\infty,  
\qquad\mbox{and}\qquad 
\tr \left(e^{-t\cR_M}\right) =  
\vol (M) p_1(0) \, t^{-Q/\nu}+O(t)^\infty. 
$$
The Weyl law is classically deduced from the latter:
the spectral counting function of $\cR_M$, 
$$
N(\Lambda):=\left|\left\{ j\in\bN_0,\;\;\lambda_j\leq \Lambda\right\}\right|,
$$
admits the asymptotic (see Theorem \ref{thm_Weyl})
$$
N(\Lambda) \sim \frac{\vol (M) p_1(0)}{\Gamma(1+Q/\nu)} \Lambda^{Q/2}, \qquad \mbox{as} \ \Lambda\to +\infty.
$$

As mentioned above, 
before  \cite{Dave+Haller}, 
most of these results regarding the heat asymptotics were known only for the intrinsic sub-Laplacians on step-two nilpotent Lie groups from the works of Bauer, Furutani and Iwasaki (see \cite{Bauer+,Bauer+21} and references therein). These works use the very explicit descriptions of the spectra of the sub-Laplacians involved \cite{Folland2004,thangavelu2009}
and known formulae for the heat kernel $p_t$ \cite{gaveau77,Beals+Gaveau+Greiner}. 
Here, we show that these properties hold not only for the heat kernels of sub-Laplacians on step-two nilmanifolds, but more generally for spectral multipliers in any Rockland operators on any graded nilmanifolds. 
Already in the step-two case, Bauer, Furutani and Iwasaki
noticed  \cite{Bauer+} that the constants in their results did not depend directly on the compact subgroup or the sub-Laplacian. Our proof explains in greater generality how and why the constants depend only on the volume of $M$, the homogeneous dimension $Q$ of $G$ and on the operator via its  heat kernel $p_1(0)$ at $x=0$ or equivalently the constant $c_0$ of Theorem \ref{thm_main} (see also Theorem \ref{thm_christ}). 

At its heart, our proof relies on the same analytic tools as the groupoid techniques: the analysis of kernels and notions of dilations that started fifty years ago, see e.g. \cite{folland+stein_82}.
However, our presentation is more direct and our results are more precise and general than in \cite{Dave+Haller} since our setting is restricted to nilmanifolds. 
To obtain our more general result on $\cR$-spectral multipliers, 
we will make use of famous results that hold in the group setting:
for instance,
the convolution kernel $\kappa=\psi(\cR)\delta_0$ of $\psi(\cR)$ being Schwartz on $G$ is a well-known theorem due to Hulanicki  \cite{hulanicki}, and the expression for $\kappa(0)$ given in Theorem \ref{thm_main} is essentially a consequence of a result due to Christ \cite{Christ91} (see  Theorem \ref{thm_christ} and Corollary \ref{cor_thm_christ}). 

\medskip

This paper is organised as follows. 
After preliminaries on nilmanifolds in Section \ref{sec_preliminar}
and on Rockland operators in Section \ref{sec_GR}, we 
discuss asymptotics  in Section \ref{sec_asymptotics}. The last section is devoted to the spectral zeta function. 

\medskip
 
 \noindent\textbf{Convention:}
In this paper, $\Gamma$ will denote a discrete co-compact subgroup of a nilpotent Lie group $G$. We will also use the same Greek letter to denote the classical Gamma function $\Gamma(s)$. As the natures of the two objects are different, there will be no confusion. 
We will use the following classical properties of the Gamma function without mentioning them.  $\Gamma(s)$ is a meromorphic function on $\bC$ with poles at $0,-1,-2,\ldots$ It satisfies 
$s\Gamma(s)=\Gamma(s+1)$ and 
\begin{equation}
\label{eq_lambda-sGamma}
\lambda ^{-s} = \frac1{\Gamma(s)} \int_0^\infty t^{s-1} e^{-t\lambda} dt, 
\quad \lambda>0, \ \Re s>0. 
\end{equation}
 Furthermore, $1/\Gamma(s)$ is entire, i.e. holomorphic on $\{s\in \bC\}$.

\section{Preliminaries on nilmanifolds, periodicity and periodisation}
\label{sec_preliminar}

\subsection{Definitions and examples of nilmanifolds}

A compact nil-manifold is the quotient $M=\Gamma\backslash G$ of 
a  nilpotent Lie group $G$  by a discrete co-compact subgroup $\Gamma$ of $G$.
In this paper, a nilpotent Lie group is always assumed  connected and simply connected unless otherwise stated. 

The vocabulary varies, and a discrete co-compact subgroup may also be called uniform or lattice in some literature, see e.g.  \cite[Section 5]{corwingreenleaf}. 
{In this paper, we will use only the vocabulary `discrete co-compact subgroup', although in the next section we will mention a connection with lattices of the Lie algebra of $\fg$. The word `lattice' will be reserved for a discrete co-compact subgroup of a finite dimensional vector space viewed as a commutative group.}

A concrete example of discrete co-compact subgroup is the natural discrete subgroup of the Heisenberg group, as described in 
\cite[Example 5.4.1]{corwingreenleaf}.
Further examples in step 2 nilpotent Lie groups can be found in \cite{Bauer+} and references therein. 
Abstract constructions for graded groups are discussed in Section \ref{subsec_Gamma4G}, and will use the following statement which   recalls  some abstract examples and characterisations from   \cite[Section 5.1]{corwingreenleaf} (the definition of Malcev bases is recalled below):

\begin{theorem}
\label{thm_CG}	
Let $G$ be a  nilpotent Lie group. Let $n=\dim G$. 
 \begin{enumerate}
\item 	A  subgroup $\Gamma$ of $G$ is discrete co-compact if and only if ${\rm Exp} Y_j \in \Gamma$ for $j=1,\ldots,n$ 
for some weak or strong Malcev basis $Y_1,\ldots,Y_n$ of $\fg$.
\item  A subgroup $\Gamma$ of $G$ is discrete co-compact  if and only if it can be written as 
$$
\Gamma={\rm Exp} (\bZ Y_1)  \ldots {\rm Exp} (\bZ Y_n)
$$
for some weak or strong Malcev basis $Y_1,\ldots,Y_n$ of $\fg$.
\end{enumerate}

Moreover, a group $G$ admits a uniform subgroup $\Gamma$ if and only if its Lie algebra $\fg$ has a rational structure. If this is the case, then a choice of rational structure  $\fg_{\mathbb Q}$ is the $\mathbb Q$-span of $\log \Gamma$.  

Furthermore, if $Y_1,\ldots,Y_n$ is a strong Malcev basis whose structural constants $c_{i,j,k}$ from $[Y_i,Y_j] = \sum_{k} c_{i,j,k} Y_k$ are all rational, 
then there exists a positive integer $K\in \bN$ such that the set 
$$
{\rm Exp} (K\bZ Y_1)  \ldots {\rm Exp} (K\bZ Y_n)
$$
 is a discrete co-compact subgroup of $G$.
\end{theorem}

Let us recall the notion of Malcev bases:
an (ordered) basis $Y_1,\ldots, Y_n$ of $\fg$ is a strong (resp. weak) Malcev basis when for each $m=1,\ldots,n$, the subspace $\bR Y_1 \oplus \ldots \oplus \bR Y_m$ is a Lie sub-algebra (resp. ideal) of the Lie algebra $\fg$.
  We refer the reader to \cite{corwingreenleaf} for examples and  properties  of these bases, and the reader unfamiliar with this notion can just consider this as a technical property satisfied by important bases, for instance by the basis constructed in  Section \ref{subsubsec_dilations} in the case of graded Lie groups.

 \subsection{Discrete co-compact subgroups}

There is a close connection between discrete co-compact subgroups and lattices in $\fg$ described in
\cite[Section 5.4]{corwingreenleaf}:

\begin{theorem}
\label{thm_InclusionLattice}
Let $\Gamma$ be a discrete co-compact subgroup of a nilpotent Lie group $G$.
Then there exists $\Gamma_0$ and $\Gamma_1$ discrete co-compact subgroups of $G$ such that 
\begin{itemize}
\item $\log \Gamma_0$ and $\log \Gamma_1$	are lattices of the vector space $\fg \sim \bR^n$, 
\item the inclusions $\Gamma_0 \subset \Gamma \subset \Gamma_1$ hold, and
\item $\Gamma/\Gamma_0$ and $\Gamma_1/\Gamma_0$ are finite sets. 
\end{itemize}
Furthermore, having written $\Gamma$ as ${\rm Exp} (\bZ Y_1)  \ldots {\rm Exp} (\bZ Y_n)$	
for some strong Malcev basis $Y_1,\ldots,Y_n$ of $\fg$, 
we may choose $\Gamma_0={\rm Exp} (K\bZ Y_1)  \ldots {\rm Exp} (K\bZ Y_n)$	for some suitable integer $K>0$, 
and $\Gamma_1={\rm Exp} (k_1\bZ Y_1)  \ldots {\rm Exp} (k_n\bZ Y_n)$ with $k_j=K^{-N^{n-j}}$, $j=1,\ldots, n$.
\end{theorem}

\begin{corollary}
\label{corthm_InclusionLattice}
Let $\Gamma$ be a discrete co-compact subgroup of a nilpotent Lie group $G$. 
We identify $G$ with $\bR^n$ via the exponential mapping and a basis of $\fg$.
	Then for any $N>n$ and any norm $|\cdot|$ on the vector space $\bR^n\sim G$, the sum 
	$\sum_{\gamma \in \Gamma } (1+|\gamma|)^{-N}$  is finite.
\end{corollary}
\begin{proof}
We may assume that the basis of $\fg$ is   a strong Malcev basis $Y_1,\ldots,Y_n$  such that 
$\Gamma={\rm Exp} (\bZ Y_1)  \ldots {\rm Exp} (\bZ Y_n)$.
Theorem \ref{thm_InclusionLattice} implies that we may assume that $\Lambda =\log \Gamma$ is a lattice of $\fg\sim \bR^n$.
In this case, 
$$
\sum_{\gamma \in \Gamma } (1+|\gamma|)^{-N}
=
\sum_{m\in \Lambda } (1+|m|)^{-N}.
$$	
The result follows by comparisons with integrals on $\bR^n$. Indeed, we can fix a suitable norm on $\bR^n$ as all the norms on $\bR^n$ are equivalent. 
\end{proof}

\subsection{Fundamental domains}
An element of $M$ is a class 
$$
\dot x := \Gamma x
$$
 of an element $x$ in $G$. If the context allows it, we may identify this class with its representative $x$. 

The quotient $M$ is naturally equipped with the structure of a compact smooth manifold. 
Furthermore, fixing a Haar measure on the unimodular group $G$, 
$M$ inherits a measure $d\dot x$ which is invariant under the translations  given by
$$
\begin{array}{rcl}
M & \longrightarrow & M\\
\dot x & \longmapsto & \dot x g = \Gamma xg
\end{array}, \quad g\in G.
$$
Recall that the Haar measure $dx$ on $G$ is unique up to a constant and, once it is fixed, $d\dot x$ is the only $G$-invariant measure on $M$ satisfying 
for any  function $f:G\to \mathbb C$, for instance continuous with compact support,
\begin{equation}
\label{eq_dxddotx}
	\int_G f(x) dx = \int_M \sum_{\gamma\in \Gamma} f(\gamma x) \ d\dot x.
\end{equation}
The volume of $M$ is denoted by 
$$
\vol (M) := \int_M 1 d\dot x
$$ 
 
Some `nice' fundamental domains are described in \cite[Section 5.3]{corwingreenleaf}, and  simple modifications from  Theorem 5.3.1 in \cite{corwingreenleaf} and its proof yield:

\begin{proposition}
\label{prop_FundDom}
	Let $\Gamma$ be a discrete co-compact subgroup of a nilpotent Lie group $G$ described as 
	$\Gamma={\rm Exp} (\bZ Y_1)  \ldots {\rm Exp} (\bZ Y_n)$	
for some weak Malcev basis $Y_1,\ldots,Y_n$ of $\fg$ (see Theorem \ref{thm_CG}.

We set $R_0 := [-\frac 12,\frac 12)\times \ldots\times [-\frac 12,\frac 12)=[-\frac 12,\frac 12)^n$ and for every $m\in \bR^n$:
$$
R_m:=m+R_0 \quad\mbox{and}\quad D_m:=\{{\rm Exp}(t_1 Y_1) \ldots \ {\rm Exp}(t_n Y_n) \ : \ t=(t_1,\ldots,t_n)\in R_m\}.
$$
Then $D_m$ is a fundamental domain for $M=\Gamma\backslash G$.
Furthermore, the map 
$$
\Theta:\left\{\begin{array}{rcl}
\bR^n 	& \longrightarrow & M
\\
t &\longmapsto &\Gamma {\rm Exp}(t_1 Y_1) \ldots \ {\rm Exp}(t_n Y_n)
\end{array}\right. ,
$$
 maps $R_m$ onto $M$ and the Lebesgue measure $dt$ restricted to $R_m$ to the $G$-invariant measure on $M$.
 If $t,u\in R_m$ and $\gamma \in \Gamma$ satisfy $\Theta(t)=\gamma \Theta(u)$ then $t=u$ and $\gamma=0$. 
 Furthermore, if $t\in \bR^n$ and $\gamma \in \Gamma$ satisfy $\Theta(t)^{-1}\gamma \Theta(t)=0$ then  $\gamma=0$.
\end{proposition}

\subsection{$\Gamma$-periodicity and periodisation}
Let $\Gamma$ be a discrete co-compact subgroup of a nilpotent Lie group $G$. 

We say that a function $f:G\rightarrow \mathbb C$  is  $\Gamma$-left-periodic or just  $\Gamma$-periodic  when we have 
$$
\forall x\in G,\;\;\forall \gamma\in \Gamma ,\;\; f(\gamma x)=f(x).
$$
This definition extends readily to measurable functions and to distributions.  

There is a natural one-to-one correspondence between the functions on $G$ which are $\Gamma$-periodic and the functions on $M$.
Indeed, for any map $F$ on $M$, 
the corresponding periodic function on $G$ is $F_G$ defined via
$$	 
F_G(x) := F(\dot x), \quad x\in G,
$$
while if $f$ is a $\Gamma$-periodic function on $G$, 
it defines a function $f_M$ on $M$ via
$$
f_M(\dot x) =f(x), \qquad x\in G.
$$
Naturally, $(F_G)_M=F$ and $(f_M)_G=f$.

We also  define, at least formally, the periodisation $\phi^\Gamma$ of a function $\phi(x)$ of the variable $x\in G$ by:
$$
\phi^\Gamma(x) = \sum_{\gamma \in \Gamma } \phi(\gamma x), \qquad x\in G.
$$

If $E$ is a space of functions or of distributions on $G$, then we denote by $E^\Gamma$ the space of elements in $E$ which are  $\Gamma$-periodic. 
Let us recall that $G$ is a smooth manifold which is identified with $\bR^n$ via the exponential mapping and polynomial coordinate systems. 
This leads to a corresponding Lebesgue measure on $\fg$ and the Haar measure $dx$ on the group $G$,
hence $L^p(G)\cong L^p(\bR^n)$.
This also allows us \cite[p.16]{corwingreenleaf}
to define the spaces 
$$
\cD(G)\cong \cD(\bR^n)
\quad \mbox{and}\quad  
\cS(G) \cong \cS(\bR^n)
$$
 of test functions which are smooth and compactly supported or Schwartz, 
and the corresponding spaces of distributions 
$$
\cD'(G)\cong \cD'(\bR^n)
\quad \mbox{and}\quad 
\cS'(G)\cong \cS'(\bR^n).
$$
Note that this identification with $\bR^n$ does not usually extend to the convolution: the group convolution, i.e. the operation between  two functions on $G$ defined formally via 
$$
 (f_1*f_2)(x):=\int_G f_1(y) f_2(y^{-1}x) dy,
$$
 is   not commutative in general whereas it is a commutative operation for functions on  the abelian group $\bR^n$.

We also define the set of functions
 $$
C_b^\infty(G):=\left\{f\in C^\infty(G): \sup_G|Y^\alpha f|<\infty\ \mbox{for every} \ \alpha\in \bN_0^n\right\},
$$
for some basis $Y_1,\ldots, Y_n$  of $\fg$ identified with left-invariant vector fields and 
$$
Y^{\alpha}=Y_1^{\alpha_1}Y_2^{\alpha_2}\cdots
Y_{n}^{\alpha_n}, \quad \alpha\in \bN_0^n.
$$
We check readily that 
the vector space $C_b^\infty(G)$ and its natural topology are independent of a choice of basis $Y_1,\ldots, Y_n$
and that 
$$
C^\infty(G)	^\Gamma= C_b^\infty(G)^\Gamma .
$$
Furthermore, we have:

\begin{lemma}
\label{lem_periodisation}
 The periodisation of a Schwartz  function $\phi\in \cS(G)$ is a well-defined function  $\phi^\Gamma$ in $C_b^\infty(G)^\Gamma$.
Furthermore, the map $\phi \mapsto \phi^\Gamma$ yields a surjective morphism of topological vector spaces from
$\cS(G)$ onto $C_b^\infty(G)^\Gamma$
and from
$\cD(G)$ onto $C_b^\infty(G)^\Gamma$.
 \end{lemma}
 
 \begin{proof}[Proof of Lemma \ref{lem_periodisation}]
 We first  need to set some notation.  By Theorem \ref{thm_CG},
 we may assume that $\Gamma={\rm Exp} (\bZ Y_1)  \ldots {\rm Exp} (\bZ Y_n)$	
for some strong Malcev basis $Y_1,\ldots,Y_n$ of $\fg$.
As strong Malcev bases yields polynomial coordinates, 
we may identify $G$ with $\bR^n$ via the exponential mapping:
$y=(y_1,\ldots,y_n)\mapsto {\rm Exp}(y_1 Y_1+\ldots+y_n Y_n)$.
We fix a Euclidean norm $|\cdot|$ on $\bR^n\sim G$.
Note that $|y^{-1}|=|-y|=|y|$ and that 
the Baker-Campbell-Hausdorff formula implies the following modified triangle inequality:
\begin{equation}
\label{eq_notTrIneq}
	\exists C_0>0 \qquad 
\forall a,b\in G\qquad 
1+|ab|\leq C_0 (1+|a|)^s (1+|b|)^s,
\end{equation}
where $s$ is the step of $G$. 

We first   show that the periodisation of a Schwartz function $\phi\in \cD(G)$ makes sense as a function on $G$. As $\phi$ is Schwartz, for all $N\in \bN$ there exists $C=C_{\phi,N}$ such that 
$$
\forall x\in G\qquad |\phi(x)|\leq C (1+|x|)^{-N}, 
 $$
so by \eqref{eq_notTrIneq},
 $$
\forall x\in G, \ \gamma\in \Gamma\qquad |\phi(\gamma x)|\leq C 
(1+|\gamma x|)^{-N} 
\leq CC_0^{N/s} (1+|x|)^{N/s}  (1+|\gamma |)^{-N/s} .
 $$
The sum  $\sum_{\gamma\in \Gamma}(1+|\gamma |)^{-N/s}$
 is finite 
 by Corollary \ref{corthm_InclusionLattice}
  for $N$ large enough (it suffices to have $N>ns$).
 Hence the function $\phi^\Gamma$ is well defined on $G$. Furthermore, 
it is now a routine exercise to check that $\phi^\Gamma \in C_b^\infty(G)^\Gamma$ and that $\phi \mapsto \phi^\Gamma$ is  a morphism of topological vector spaces between $\cS(G)$   to $C_b^\infty(G)^\Gamma$, and also from  $\cD(G)$  to $C_b^\infty(G)^\Gamma$.
 	
 It remains to show the surjectivity of $\cD(G)\ni \phi\mapsto \phi^\Gamma \in C_b^\infty(G)^\Gamma$.
We observe
\begin{equation}
\label{eq_lem_periodisation_phipsi}
	\forall \phi \in C_b^\infty(G)^\Gamma, \ \forall \psi \in \cD(G)
\qquad \phi \, \psi \in \cD (G) \quad\mbox{and}\quad 
(\phi \psi)^\Gamma = \phi \, \psi^\Gamma.
\end{equation}
We fix $\psi_0\in \cD(G)$ valued in $[0,1]$ and such that $\psi_0=1$ on a fundamental domain of $M$, for instance the fundamental domain $D_0$ from Proposition \ref{prop_FundDom} to fix the ideas. We observe that $\psi_0^\Gamma\in 
C_b^\infty(G)^\Gamma$ {satisfies $\psi_0 \geq 1$}, and furthermore that $1/\psi_0^\Gamma$ is also in $C_b^\infty(G)^\Gamma$.
Given any $\phi \in C_b^\infty(G)^\Gamma$, 
applying \eqref{eq_lem_periodisation_phipsi}
to $\phi$ and $\psi = \psi_0 / \psi_0^\Gamma$ shows the surjectivity of $\cD(G)\ni \phi\mapsto \phi^\Gamma \in C_b^\infty(G)^\Gamma$ and concludes the proof of Lemma \ref{lem_periodisation}. 
 \end{proof}

\subsection{Spaces of periodic functions}
\label{subsec_periodicfcn}

We now examine $E^\Gamma$ for some spaces of functions $E$ on the nilpotent Lie group $G$, 
where  $\Gamma$ is a discrete co-compact subgroup of a nilpotent Lie group $G$.
Although 
$$
\cD(G)^\Gamma = \{0\} = \cS(G)^\Gamma,
$$
many other $E^\Gamma$ are isomorphic to important spaces of functions (or distributions) on $M$ as the following lemmata suggest. 

The definition of $ F_G$ and $f_M$  extend to measurable functions and we have:
\begin{lemma}
\label{lem_isomorphism_int}
For every $p\in [1,\infty]$,
the map $F\mapsto F_G$ is an isomorphism of topological vector spaces (in fact Banach spaces) from $L^p(M)$ onto $L^p_{loc}(G)^\Gamma$ with inverse $f\mapsto f_M$.
 \end{lemma}
The proof follows readily from the description of fundamental domains above, see Proposition \ref{prop_FundDom}. It is left to the reader.

We also check readily:
\begin{lemma}
\label{lem_isomorphism_smooth}
The mapping $F\mapsto F_G$ is an isomorphism of topological vector spaces from 
$\cD(M)  $ onto 
$ C_b^\infty(G)^\Gamma $  with inverse $f\mapsto f_M$.
\end{lemma}

Consequently, $\cD'(M)  $ is isomorphic to the dual of 
$C_b^\infty(G)^\Gamma $. This allows for a first distributional meaning to $F\mapsto F_G$ 
with $F_G$ in the continuous dual of $C_b^\infty(G)^\Gamma$ and extended to $C_b^\infty(G)$ by Hahn-Banach's theorem.
However, we prefer to extend the definition of $F_G$ to the case of distributions in the following way: if $F\in \cD'(M)$, then $F_G$ given by 
$$
\forall \phi\in \cS(G)\qquad 
\langle F_G,\phi\rangle  = \langle F , (\phi^\Gamma)_M\rangle.
$$
is a tempered distribution by 
Lemmata \ref{lem_isomorphism_smooth} and \ref{lem_periodisation}. One checks easily that it is periodic and that it coincides with any other definition given above, for instance on  $ \cup_{p\in [1,\infty)} L^p(M)$. Furthermore, we have:
\begin{lemma}
\label{lem_isomorphism_distrib}
 We have $\cD'(G)^\Gamma=\cS'(G)^\Gamma$,
 and the map $F\mapsto F_G$ yields an isomorphism of topological vector spaces
from $\cD'(M)$ onto $\cS'(G)^\Gamma$.
 \end{lemma}
 
 \begin{proof}
 	Lemmata \ref{lem_periodisation} and \ref{lem_isomorphism_smooth} imply easily that $F\mapsto F_G$ is a morphism of topological vector spaces
from $\cD'(M)$ to $\cS'(G)^\Gamma$.
We can construct its inverse using the function $\psi_0\in \cD(G)$ from the proof of Lemma \ref{lem_periodisation}.
If $f\in \cD'(G)^\Gamma$ then we define
$$
\forall \psi\in \cD(M) \qquad \left\langle f_M,\psi \right\rangle := \left\langle f, \psi_G \frac {\psi_0}{\psi_0^\Gamma}\right\rangle. 
$$
Lemma \ref{lem_isomorphism_smooth} implies easily that $f_M$ is a distribution on $M$ and that $f\mapsto f_M$ is a morphism of topological vector spaces
from  $\cD'(G)^\Gamma$ to $\cD'(M)$.
Furthermore, it gives the inverse of $F\mapsto F_G$ since we have first from the definitions of these two mappings:
$$
\forall \phi \in \cD(G)\qquad
\left\langle (f_M)_G,\phi\right\rangle 
= \left\langle f_M,(\phi^\Gamma)_M \right \rangle
= \left\langle f,\phi^\Gamma  \frac {\psi_0}{\psi_0^\Gamma}\right \rangle
= \left\langle f,\phi \right\rangle,
$$
by periodicity of $f$. Hence $f=(f_M)_G \in \cS'(G)^\Gamma$ for any $f\in 
\cD'(G)^\Gamma$. The statement follows. 
 \end{proof}
 
One checks easily that the inverse $f\mapsto f_M$  of $F\mapsto F_G$ constructed in the proof above coincides with any other definition of $f\mapsto f_M$ given above, for instance on $ \cup_{p\in [1,\infty)} L^p_{loc}(G)$.
Moreover, for every  $p\in [1,\infty)$,  since $L^p(M)\subset \cD'(M)$ , we have  $L^p_{loc}(G)^\Gamma \subset \cS'(G)^\Gamma$
by Lemma \ref{lem_isomorphism_distrib}. 

\subsection{Convolution and periodicity}
\label{subsec_conv+per}
We already know that the convolution of a tempered distribution with a Schwartz function is smooth and bounded on a nilpotent Lie group $G$.
 When the distribution is periodic under the discrete co-compact subgroup of $G$, we also have the following properties, in particular a type of Young's convolution inequality:

\begin{lemma}
\label{lem_convolution}
Let $f\in \cS'(G)^\Gamma$ and $\kappa\in \cS(G)$.
Then $f*\kappa \in C_b^\infty(G)^\Gamma$.
Viewed as a function on $M$, 
$$
(f*\kappa)_M(\dot x) = \int_M f_M(\dot y) \ (\kappa (\cdot ^{-1} x)^\Gamma)_M (\dot y) d\dot y
= \int_M f_M(\dot y) \sum_{\gamma\in \Gamma} \kappa (y^{-1} \gamma x) \  d\dot y.
$$

If $F\in L^p(M)$ for $p\in [1,+\infty]$, then $F_G\in \cS'(G)^\Gamma$  and we have
$$
\|	(F_G*\kappa)_M\|_{L^p(M)}
\leq
\|F\|_{L^p( M)}\|\kappa\|_{L^1(G)}
$$ 
\end{lemma}
\begin{proof}
We check readily for $x\in G$ and $\gamma\in \Gamma$
$$
f*\kappa (\gamma x)  =
\int_G f(y) \kappa(y^{-1} \gamma x) dy= \int_G f(\gamma z) \kappa (z^{-1}x)dz
= \int_G f(z) \kappa (z^{-1}x)dz= f*\kappa (x).
$$ 	
The formula on $M$ follows from \eqref{eq_dxddotx}.

Let $F\in L^p(M)$ for $p\in [1,+\infty]$.
As a consequence of Lemmata \ref{lem_isomorphism_int} and \ref{lem_isomorphism_distrib},
 $F_G\in \cS'(G)^\Gamma \cap L^p_{loc}(G)^\Gamma$.
By Lemmata \ref{lem_periodisation} and \ref{lem_isomorphism_smooth}, for each fixed $\dot x\in M$, we can set
$$
d_{\dot x}(\dot y):=\kappa (\cdot ^{-1} x)^\Gamma (y) = \sum_{\gamma\in \Gamma} \kappa (y^{-1} \gamma x),
$$ 
and this defines a smooth function $d_{\dot x}$ on $M$.
Furthermore, $\dot x\mapsto d_{\dot x}$ is continuous from $M$ to $\cD(M)$.
This function allows us to write the more concise formula
$$
 (F_G*\kappa)_M(\dot x)
   = \int_M F(\dot y) d_{\dot x}(\dot y) d\dot y.
 $$
 
 The decomposition of the Haar measure in \eqref{eq_dxddotx} and its  invariance  under translation imply
\begin{align}
\label{eq_pf_lem_convolution_xfixed}
	\|d_{\dot x}\|_{L^1(M)} \leq 
\int_M  \sum_{\gamma\in \Gamma} |\kappa (y^{-1} \gamma x)| \  d\dot y
=
\int_G |\kappa (y^{-1} x)| \  d y
=\|\kappa\|_{L^1(G)}, \quad \dot x\in M \ \mbox{(fixed)},
\\	
\label{eq_pf_lem_convolution_yfixed}
\int_M |d_{\dot x}(\dot y)| d\dot x
\leq 
 \int_M  \sum_{\gamma\in \Gamma} |\kappa (y^{-1} \gamma x)| \  d\dot x
=
\int_G |\kappa (y^{-1} x)| \  d x
=\|\kappa\|_{L^1(G)}, \quad \dot y\in M \ \mbox{(fixed)}.
\end{align}

The case $p=+\infty$ follows from \eqref{eq_pf_lem_convolution_xfixed} since we have
$$
 	|(F_G*\kappa)_M(\dot x)|
     \leq \|F\|_{L^\infty(M)} \|d_{\dot x}\|_{L^1(M)} = \|F\|_{L^\infty(M)} \|\kappa\|_{L^1(G)}.
$$

Let $p\in [1,+\infty)$. 
By Jensen's inequality, we have for any fixed $\dot x\in M$
$$
 	|(F_G*\kappa)_M(\dot x)|^p
   = \left|\int_M F(\dot y) d_{\dot x} (\dot y) d\dot y \right|^p 
   \leq \|d_{\dot x}\|_{L^1(M)}^p \int_M \left| F(\dot y)\right|^{p-1}  |d_{\dot x} (\dot y)| d\dot y.
   $$
Integrating against $\dot x\in M$ and using \eqref{eq_pf_lem_convolution_yfixed} conclude the proof. 
\end{proof}

\subsection{Operators on $M$ and $G$}
\label{subsec_OpGM}

{If $f:G\to \bC$ is a function on $G$ and $g\in G$, we denote by  $f(g \, \cdot)$  the function $x\mapsto f(g x)$ on $G$. One checks readily that $f\mapsto f( g\, \cdot)$ is an isomorphism of topological vector spaces on $\cS(G)$ and on $\cD(G)$. By duality, the definition of $f(g \, \cdot)$ extends to distributions in $\cD'(G)$. Furthermore, $f\mapsto f( g\, \cdot)$ is an isomorphism of topological vector spaces on $\cS'(G)$ and on $\cD'(G)$.
We say that a mapping $T:\cS'(G)\to \cS'(G)$
or $\cD'(G) \to \cD'(G)$ is invariant under an element $g\in G$ when  
$$
\forall f\in \cS'(G) \ (\mbox{resp.} \, \cD'(G)), \qquad T(f(g \, \cdot)) = (Tf)(g \, \cdot).
$$ 
It is invariant under a subset of $G$ if it is invariant under every element of the subset.}

{Consider a linear continuous mapping $T:\cS'(G)\to \cS'(G)$
or $\cD'(G) \to \cD'(G)$ respectively
which is invariant under $\Gamma$. Then it  naturally induces an operator $T_M$ on $M$ via
$$
T_M F = (TF_G)_M, \qquad F\in \cD'(M).
$$
By Lemma \ref{lem_isomorphism_distrib}, $T_M :\cD'(M)\to \cD'(M)$  is 
a linear continuous mapping. }

For convolution operators, 
the results in  Section \ref{subsec_periodicfcn}, especially Lemma \ref{lem_convolution}, yield:
\begin{lemma}
\label{lem_IntKernel}
Let $\kappa\in \cS(G)$ be a given convolution kernel, and let us denote by $T$ the associated convolution operator:
$$
T (\phi) = \phi*\kappa, \qquad \phi\in \cS'(G).
$$
The operator $T$ is a linear continuous mapping $\cS'(G)\to \cS'(G)$. 
The corresponding operator $T_M$ maps  $\cD'(M)$ to $\cD'(M)$ continuously and linearly. Its  integral kernel is the smooth function $K$ on $M\times M$ given by
$$
K(\dot x,\dot y) = 
\sum_{\gamma\in \Gamma}  \kappa(y^{-1}\gamma  x).
$$
Consequently, the operator $T_M$ is Hilbert-Schmidt on $L^2(M)$ with Hilbert-Schmidt norm 
$$
\|T_M\|_{HS} = \|K\|_{L^2(M\times M)}.
$$
\end{lemma}

Invariant differential operators keep many of their features from $G$ to $M$:
\begin{lemma}
Let $T$ be a linear continuous map $\cS'(G)\to \cS'(G)$ or $\cD'(G)\to \cD'(G)$.
We assume that $T$ coincide with   a smooth differential operator 	on $G$ that is invariant under $\Gamma$.
Then $T_M$ is a smooth differential operator on $M$.

If $T$ is hypo-elliptic, then so is $T_M$.

If $T$ is symmetric and positive on $\cS(G)$ in the sense that 
$$
\forall F_1,F_2\in \cS(G) \quad \int_G TF_1(x)\, \overline{ F_2(x)} dx = \int_G F_1(x)\, \overline{T F_2(x)} dx
\quad\mbox{and}\quad
\int_G TF_1(x)\, \overline{ F_1(x)} dx \geq 0, 
$$
then $T_M$ is also symmetric and positive on $\cD(M)$.
Moreover, 
$T$ is essentially self-adjoint on $L^2(G)$ and $T_M$ is essentially self-adjoint on $L^2(M)$.
\end{lemma}

These arguments are very classical and we only sketch their proofs.  

\begin{proof}[Sketch of the proof]
	The manifold $M$ have natural charts given by the descriptions of the fundamental domains in Proposition \ref{prop_FundDom}.
They imply that  $T_M$ is a smooth differential operator on $M$ and  also that if $T$ is hypo-elliptic (resp. symmetric and positive), then so is $T_M$.
The self-adjointness follows from the fact that the positivity 
imply the density of the ranges of the operators $T+ i {\rm I}$ and $T-i{\rm I}$, and $T_M+ i {\rm I}$ and $T_M-i{\rm I}$.
\end{proof}

\section{Positive Rockland operators on nilmanifolds}
\label{sec_GR}

In this section, we review positive Rockland operators on  graded Lie groups and their corresponding operators on the associated nilmanifolds.  

\subsection{Preliminaries on graded nilpotent Lie group}

In the rest of the paper, 
we will be concerned with graded Lie groups.
References on this subject includes \cite{folland+stein_82} and 
\cite{R+F_monograph}.

A graded Lie group $G$  is a connected and simply connected 
Lie group 
whose Lie algebra $\fg$ 
admits an $\bN$-gradation
$\fg= \oplus_{\ell=1}^\infty \fg_{\ell}$
where the $\fg_{\ell}$, $\ell=1,2,\ldots$, 
are vector subspaces of $\fg$,
almost all equal to $\{0\}$,
and satisfying 
$[\fg_{\ell},\fg_{\ell'}]\subset\fg_{\ell+\ell'}$
for any $\ell,\ell'\in \bN$.
This implies that the group $G$ is nilpotent.
Examples of such groups are the Heisenberg group
 and, more generally,
all stratified groups (which by definition correspond to the case $\fg_1$ generating the full Lie algebra $\fg$).

\subsubsection{Dilations and homogeneity}
\label{subsubsec_dilations}
For any $r>0$, 
we define the  linear mapping $D_r:\fg\to \fg$ by
$D_r X=r^\ell X$ for every $X\in \fg_\ell$, $\ell\in \bN$.
Then  the Lie algebra $\fg$ is endowed 
with the family of dilations  $\{D_r, r>0\}$
and becomes a homogeneous Lie algebra in the sense of 
\cite{folland+stein_82}.
We re-write the set of integers $\ell\in \bN$ such that $\fg_\ell\not=\{0\}$
into the increasing sequence of positive integers
 $\upsilon_1,\ldots,\upsilon_n$ counted with multiplicity,
 the multiplicity of $\fg_\ell$ being its dimension.
 In this way, the integers $\upsilon_1,\ldots, \upsilon_n$ become 
 the weights of the dilations.

 We construct a basis $X_1,\ldots, X_n$  of $\fg$ adapted to the gradation,
by choosing a basis $\{X_1,\ldots, X_{n_1}\}$ of $\fg_1$ (this basis is possibly reduced to $\emptyset$), then 
$\{X_{n_1+1},\ldots,  X_{n_1+n_2}\}$ a basis of $\fg_2$
(possibly $\emptyset$ as well as the others).
We have $D_r X_j =r^{\upsilon_j} X_j$, $j=1,\ldots, n$.
We may identify $G$ with $\bR^n$ via the exponential mapping:
$x=(x_1,\ldots,x_n)\mapsto {\rm Exp}(x_1 X_1+\ldots+x_n X_n)$.

 The associated group dilations are defined by
$$
D_r(x)=
rx
:=(r^{\upsilon_1} x_{1},r^{\upsilon_2}x_{2},\ldots,r^{\upsilon_n}x_{n}),
\quad x=(x_{1},\ldots,x_n)\in G, \ r>0.
$$
In a canonical way,  this leads to the notions of homogeneity for functions, distributions and operators and we now give a few important examples. 

The Haar measure is $Q$-homogeneous where 
$$
Q:=\sum_{\ell\in \bN}\ell \dim \fg_\ell=\upsilon_1+\ldots+\upsilon_n,
$$
 is called the homogeneous dimension of $G$.

Identifying the elements of $\fg$ with left-invariant vector fields, 
each  $X_j$ is a homogeneous differential operator of degree $\upsilon_j$. 
More generally, the differential operator 
$$
X^{\alpha}=X_1^{\alpha_1}X_2^{\alpha_2}\cdots
X_{n}^{\alpha_n}, \quad \alpha\in \bN_0^n$$
is homogeneous with degree
$$
[\alpha]:=\alpha_1 \upsilon_1 + \cdots + \alpha_n \upsilon_n.
$$

 \begin{remark}
 \label{rem_natural_dilation}
 The natural dilation associated with the gradation $\fg= \oplus_{\ell=1}^\infty \fg_{\ell}$ is the one defined above via $D_r X=r^\ell X$ for every $X\in \fg_\ell$, $\ell\in \bN$.
 Note that we could also have defined other dilations, for instance
 $D_r X=r^{m \ell} X$ for every $X\in \fg_\ell$, $\ell\in \bN$, where $m$ is a fixed positive number. 
 If $m$ is an integer, this means considering the new gradation 
 $\fg= \oplus_{\ell=1}^\infty \tilde \fg_{\ell}$ with $\tilde \fg_{m \ell} = \fg_\ell$ and $\tilde \fg_{\ell'}=0$ when $\ell'$ is not a multiple of $\alpha$. 

Note that  the homogeneous dimension and homogeneous degrees for this new gradation are $\tilde Q=m Q$ and $[\alpha]\,\tilde{}=m[\alpha]$, so that the ratio $Q / [\alpha]$ is in fact independent of $m$. 
 Furthermore, choosing $m$ even implies that we may choose any degree $[\alpha]$ to be even. 
  \end{remark}

\subsubsection{Homogeneous quasi-norms}
 
An important class of homogeneous maps are the homogeneneous quasi-norms, 
that is, a $1$-homogeneous non-negative map $G \ni x\mapsto \|x\|$ which is symmetric and definite in the sense that $\|x^{-1}\|=\|x\|$ and $\|x\|=0\Longleftrightarrow x=0$.
In fact, all the homogeneous quasi-norms are equivalent in the sense that if $\|\cdot\|_1$ and $\|\cdot\|_2$ are two quasi-norms, then 
$$
\exists C>0 \qquad \forall x\in G
\qquad C^{-1} \|x\|_1 \leq \|x\|_2 \leq C \|x\|_1.
$$
Examples may be constructed easily, for instance 
$$
\|x\| = (\sum_{j=1}^n |x_j|^{N/\upsilon_j})^{-N} \ \mbox{for any}\ N\in \bN,
$$
 with the convention above.

An important property of homogeneous quasi-norms is that they satisfy the triangle inequality up to a constant:
$$
\exists C>0 \qquad \forall x,y\in G
\qquad \|xy\| \leq C (\|x\|+\|y\|).
$$
However,   it is possible to construct a homogeneous quasi-norm on $G$ which  yields a distance on $G\sim \bR^n$ in the sense that  the triangle inequality above is satisfied with constant 1 \cite[Theorem 3.1.39]{R+F_monograph}.

\subsection{Discrete co-compact subgroups of $G$}
\label{subsec_Gamma4G}
If the structural constants $c_{i,j,k}$ from $[X_i,X_j] = \sum_{k} c_{i,j,k} X_k$ are all rational, then 
then there exists a positive integer $K\in \bN$ such that the set 
$$
\Gamma :=  {\rm Exp} (K\bZ  X_n)  \ldots {\rm Exp} (K\bZ  X_1)=  {\rm Exp} (K\bZ  X_1)  \ldots {\rm Exp} (K\bZ  X_n)
$$	
 is a discrete co-compact subgroup of $G$ as a consequence of Theorem \ref{thm_CG} and the following statement:
 
   \begin{lemma}
  \label{lem_Xj_Mbasis}
  Let $G$ be a graded Lie group.
  The basis constructed in Section \ref{subsubsec_dilations} but ordered as $X_n,\ldots, X_1$ is a strong Malcev basis of $\fg$.	
  \end{lemma}
 \begin{proof}[Proof of Lemma \ref{lem_Xj_Mbasis}]
The properties of the dilations implies that the Lie bracket of an element $X\in \fg$ with $X_m$ is in the linear span of $X_j$'s with weights $>m$, hence in $\bR X_{m+1} \oplus \ldots \oplus \bR X_n$.
 \end{proof}

\begin{remark}
\begin{itemize}
\item In what follows, we do not require that $\Gamma$ is constructed out of a basis constructed in Section \ref{subsubsec_dilations}.
\item Although we will not need this property, we observe that the same proof yields the fact that $X_n,\ldots, X_1$ is a strong Malcev basis of $\fg$ through the sequence of ideals $\fg_{(1)}\subset \fg_{(2)} \subset \ldots \subset \fg_{(k)} =\fg$ defined as follows. 
Re-labeling the weights as a sequence of strictly decreasing integers
$$
\{1\leq \upsilon_1 \leq \ldots \leq \upsilon_n\}
=\{\upsilon_n=\omega_1 > \ldots > \omega_k =\upsilon_1\},
$$
$\fg_{(j)}$ denotes
the vector space  spanned by the elements in $\fg$  with weights $\geq \omega_j$ for each $j=1,\ldots, k$.
\end{itemize}
\end{remark}

We will need the following property which is similar to Corollary \ref{corthm_InclusionLattice} but for homogeneous quasi-norms:

\begin{lemma}
	\label{lem_sum}
Let $\Gamma$ be a discrete co-compact subgroup of a graded group $G$. 
	Then for any $N>\upsilon_n n$ and any homogeneous quasi-norm $\|\cdot\|$ on $G$, the sum 
	$\sum_{\gamma \in \Gamma\setminus\{0\} } \|\gamma\|^{-N}$ finite.
\end{lemma}
\begin{proof}
As all the homogeneous quasi-norms are equivalent, 
it suffices to prove the result for the one given by $\|x\| = \max_{j=1,\ldots,n} |x_j|^{1/\upsilon_j}$, where we have written $x={\rm Exp} \sum_{j=1}^n x_j X_j$, for the basis $X_1,\ldots, X_n$ adapted to the gradation and constructed in Section \ref{subsubsec_dilations}.
We observe that it suffices to show that the sum over $\gamma\in \Gamma$ with $\|\gamma\|\geq 1$ is finite, and that 
for any $x\in G$
$$
\|x\|\geq 1\Longrightarrow
\|x\| \geq (\max_{j=1,\ldots,n} |x_j|)^{1/\upsilon_n}.
$$
Now $x\mapsto \max_{j=1,\ldots,n} |x_j|$ is a norm on $\bR^n$, 
and we can conclude with Corollary \ref{corthm_InclusionLattice}.
\end{proof}

\subsection{Positive Rockland operators on $G$}
Let us briefly review the definition and main properties of positive Rockland operators. 
References on this subject includes \cite{folland+stein_82} and 
\cite{R+F_monograph}.

\subsubsection{Definitions}
\label{subsubsec_posRdef}

A \emph{Rockland operator}
 $\cR$ on $G$ is 
a left-invariant differential operator 
which is homogeneous of positive degree and satisfies the Rockland condition, that is, 
for each unitary irreducible continuous representation $\pi$ on $G$,
except for the trivial representation, 
the operator $\pi(\cR)$ is injective on the space  $\mathcal H_\pi^\infty$ of smooth vectors of the infinitesimal representation;
we have kept the same notation for the corresponding infinitesimal representation
which acts on the universal enveloping algebra $\fU(\fg)$ of the Lie algebra of the group.

Recall \cite{HelfferNourrigat79}
that Rockland operators are hypoelliptic.
In fact, they may  be equivalently characterised as the left-invariant homogeneous differential operators which are hypoelliptic. 
If this is the case, then $\cR + \sum_{[\alpha]< \nu}c_\alpha X^\alpha$ where $c_\alpha\in \bC$ and $\nu$ is the homogeneous degree of $\cR$ is hypoelliptic.

A Rockland operator is \emph{positive} when 
$$
\forall f \in \cS(G),\qquad
\int_G \cR f(x) \ \overline{f(x)} dx\geq 0.
$$

Any sub-Laplacian with the sign convention 
\begin{equation}
\label{eq_cL}
\cL=-(Z_1^2+\ldots+Z_{n'}^2),
\end{equation}
 of a stratified Lie group  is a positive Rockland operator; here $Z_1,\ldots,Z_{n'}$ form a basis of the first stratum $\fg_1$.
The reader familiar with the Carnot group setting may 
 view positive Rockland operators as generalisations of the canonical sub-Laplacians on the Carnot groups.
Positive Rockland operators are easily constructed on any graded Lie group \cite[Section 4.2]{R+F_monograph}.

A positive Rockland operator is essentially self-adjoint on $L^2(G)$ and we keep the same notation for its self-adjoint extension.
Its spectrum is included in $[0,+\infty)$ and the point 0 may be neglected in its spectrum.

\subsubsection{Spectral multipliers in $\cR$}

If $\psi:\bR^+\to \bC$ is a measurable function,
the spectral multiplier $\psi(\cR)$ is well defined as an  operator (possibly unbounded) on $L^2(G)$.
Its domain  is the space of function $\psi \in L^2(G)$ such that $\int_0^{+\infty} |\psi(\lambda)|^2 d(E_\lambda \psi,\psi)$ is finite
where $E$ denotes the spectral measure of $\cR= \int_0^{+\infty} \lambda dE_\lambda.$
For instance, if the function $\psi$ is bounded on $\bR^+$, then the operator $\psi(\cR)$ is bounded on $L^2(G)$.
However, we will be concerned with multipliers $\psi$ that may not be in $L^\infty (\bR^+)$.
If the domain of $\psi(\cR)$  contains $\cS(G)$ 
and defines a continuous map $\cS(G)\to \cS'(G)$, then 
it is invariant under right-translation; by the Schwartz kernel theorem, it admits a right-convolution kernel $\psi(\cR)\delta_0 \in \cS'(G)$ which satisfies the following homogeneity property:
\begin{equation}
\label{eq_homogeneitypsiR}	
 \psi(r^\nu \cR) \delta_0 (x) =r^{-Q} \psi(\cR)\delta_0(r^{-1}x),
 \quad x\in G.
\end{equation}

The following statement is a consequence of the famous result due to Hulanicki \cite{hulanicki}:
\begin{theorem}[Hulanicki's theorem]
\label{thm_hula}
	Let $\cR$ be a positive Rockland operator on $G$.
	If $\psi\in \cS(\bR)$ then $\psi(\cR)\delta_0 \in \cS(G)$.
\end{theorem}

For instance,  the heat kernels 
$$
p_t:=e^{-t\cR}\delta_0, \quad t>0,
$$
 are Schwartz - although this property is in fact used in the proof of Hulanicki's Theorem. 

\medskip

Hulanicki's theorem in Theorem \ref{thm_hula} and the homogeneity property in \eqref{eq_homogeneitypsiR} imply that $\psi(r\cR)$ is an approximate identity 
(see \cite[Section 3.1.10]{R+F_monograph}):
\begin{lemma}
\label{lem_appId}
Let $\psi\in \cS(\bR)$. 
Then $\psi(r\cR)$, $r>0$, is a multiple of an approximate identity in the sense that 
we have the convergence 
$\psi(r\cR) f \longrightarrow \psi(0) f$ as $r\to 0$
in $\cS(G)$ for any $f\in \cS(G)$ and in $\cS'(G)$ for any $f\in \cS'(G)$.
\end{lemma}
The  convergence also takes place in $L^p(G)$, $p\in [1,\infty)$, but we will not use this here.  

\medskip

The following result was mainly obtained by Christ for sub-Laplacians on stratified groups \cite[Proposition 3]{Christ91}. As explained below, the proof extends  to positive Rockland operators:
\begin{theorem}
\label{thm_christ}
	Let $\cR$ be a positive Rockland operator of homogeneous degree $\nu$ on $G$. 
	If the measurable function  $\psi:\bR^+\to \bC$ is in  $L^2(\bR^+,  \lambda^{Q/\nu} d\lambda/\lambda)$, 
	then the operator $\psi(\cR)$  defines a continuous map $\cS(G)\to \cS'(G)$ whose convolution kernel  $\psi(\cR)\delta_0$ is in $L^2(G)$. Moreover,   we have 
$$
\|\psi(\cR)\delta_0\|_{L^2(G)}^2
 	=c_0\int_0^\infty |\psi (\lambda)|^2  \lambda^{\frac Q \nu} \frac{d\lambda}{\lambda},
 $$
 where $c_0 = c_0(\cR)$ is a positive constant of $\cR$ and $G$.  \end{theorem}

In other words, 
the map $\psi\mapsto \psi(\cR)\delta_0$ is  an isometry from $L^2((0,\infty), c_0 \lambda^{Q/\nu} d\lambda/\lambda)$ onto its image which is a closed subspace of $L^2(G)$.
We will comment on the constant $c_0$ in 
Section \ref{subsubsec_c_0} 

\begin{proof}[Sketch of the proof of Theorem \ref{thm_christ}]
Since the heat kernels are Schwartz, we can argue as in the proof of \cite[Proposition 3]{Christ91}: for any $b>0$, the kernel 
$$
\phi_0 := 1_{(0,b]}(\cR)\delta_0 \ \mbox{is in}\ L^2(G),
$$
and, for every $\psi\in L^\infty(0,b]$, 
the kernel $\psi(\cR)\delta_0$ is in $L^2(G)$ with 
$$
\|\psi(\cR)\delta_0\|_{L^2(G)}^2
=
\int_0^\infty |\psi(\lambda)|^2 d (E_\lambda \phi_0,\phi_0).
$$
This implies the existence and uniqueness of a sigma-finite Borel measure $m$ on $\bR^+$ satisfying 
 $$
\|\psi(\cR)\delta_0\|_{L^2(G)}^2
=
\int_0^\infty |\psi(\lambda)|^2 dm (\lambda),
$$
for every $\psi\in L^\infty(\bR^+)$ with compact support. 
From the uniqueness and the homogeneity property in \eqref{eq_homogeneitypsiR}, 
it follows that the measure $m$ is homogeneous of degree $Q/\nu$ on $\bR^+$.
This means that $\lambda^{-Q/\nu} dm(\lambda)$ is a Haar measure for the multiplicative group $\bR^+$, and is therefore a constant multiple of $d\lambda /\lambda$.
This shows the theorem for any $\psi \in L^\infty(\bR^+)$ with compact support.

Let us now prove the theorem for $\psi$ in $L^2(\bR^+,  \lambda^{Q/\nu} d\lambda/\lambda)$.  
The  first problem is to ensure that $\psi(\cR)\delta_0$ makes sense. 
For this, let $(\psi_j)_{j\in \bN}$ be a sequence of bounded compactly supported functions that converges to $\psi$ in 
$L^2(\bR^+,  \lambda^{Q/\nu} d\lambda/\lambda)$ as $j\to \infty$.
As the statement is shown for the $\psi_j$'s, 
the kernels $\psi_j(\cR)\delta_0$, $j\in \bN$, form a Cauchy sequence in the Hilbert space $L^2(G)$; we denote by $\kappa \in L^2(G)$ its limit.
We observe that 
for each $\phi\in \cS(G)$, Fatou's inequality yields
$$
\int_0^\infty |\psi(\lambda) |^2 d (E_\lambda \phi,\phi) 
\leq \liminf_{j\to \infty}
\int_0^\infty |\psi_j(\lambda) |^2 d (E_\lambda \phi,\phi) 
\leq c_0^{-1}
\|\phi\|_{L^1(G)}^2 \|\kappa\|_{L^2(G)}^2 ,
$$
since we have
$$ 
c_0\int_0^\infty |\psi_j(\lambda) |^2 d (E_\lambda \phi,\phi) 
=
\|\psi_j(\cR)\phi\|_{L^2(G)}^2
=
\|\phi *\psi_j(\cR)\delta_0\|_{L^2(G)}^2
\leq 
\|\phi\|_{L^1(G)}^2 \|\psi_j(\cR)\delta_0\|_{L^2(G)}^2,
$$
by the Young convolution inequality.
This shows that $\psi(\cR)$ defines a continuous map $\cS(G)\to \cS'(G)$ thus admits a right-convolution kernel $\psi(\cR)\delta_0 \in \cS'(G)$ in the sense of distributions. 
 It remains to show the equality $\kappa=\psi(\cR)\delta_0$.
For this we observe that we have 
for any $\chi\in \cD(\bR)$ and any $\phi\in \cS(G)$:
$$
\|(\psi(\cR) -\psi_j(\cR)) \chi (\cR) \phi\|_{L^2(G)}
\leq \|(\psi -\psi_j) \chi (\cR) \delta_0\|_{L^2(G)}
 \|\phi\|_{L^1(G)},
 $$
and, since the statement is proved for  functions with compact support,
$$
\|(\psi -\psi_j) \chi (\cR) \delta_0\|_{L^2(G)}^2
= c_0\int_0^\infty|(\psi -\psi_j) \chi (\lambda)|^2 \lambda^{\frac Q \nu} \frac {d\lambda}{\lambda}
\leq c_0 \|\chi\|_{L^\infty(\bR)}
\|\psi-\psi_j\|_{L^2(\bR^+,  \lambda^{Q/\nu} d\lambda/\lambda)}.
$$ 
The last expression converges to 0 as $j\to \infty$ by hypothesis.
Hence 
$$
\lim_{j\to \infty}\|(\psi(\cR) -\psi_j(\cR)) \chi (\cR) \phi\|_{L^2(G)}=0 = \|(\chi(\cR)\phi )*(\psi(\cR)\delta_0 -\kappa )\|_{L^2(G)}.
$$
This implies that $\langle \chi(\cR)\phi , \psi(\cR)\delta_0  -\kappa\rangle=0$
and thus $\psi(\cR)\delta_0 -\kappa=0$ as tempered distributions by Lemma \ref{lem_appId}.
This concludes the sketch of the proof of Theorem \ref{thm_christ}.
\end{proof}

Bilinearising the $L^2$-norms in Theorem \ref{thm_christ} and using an approximate identity as in Lemma \ref{lem_appId} easily yield:
\begin{corollary}
\label{cor_thm_christ}
Keeping the notation of Theorem \ref{thm_christ}, we have for any $\psi\in \cS(\bR)$
$$
\psi(\cR)\delta_0(0)
 	=c_0\int_0^\infty \psi (\lambda)  \lambda^{\frac Q \nu} \frac{d\lambda}{\lambda}.
 $$
\end{corollary}

\subsubsection{Comments on the constant $c_0$}
\label{subsubsec_c_0}

Our first observation is the following expression for the constant in the statement in terms of the heat kernel $p_t$ of $\cR$:
\begin{equation}
\label{eq_c0}
	c_0 = c_0(\cR) = \frac{p_1(0)}{\Gamma(Q/\nu)}.
\end{equation}
Indeed, it suffices to plug   the function $\psi(\lambda) =e^{-\lambda}$ in Corollary \ref{cor_thm_christ}.

We observe that the operator $\cR$ being positive Rockland and the constant $c_0(\cR)$
 are independent of a choice of dilations adapted to the gradation of $G$ in the sense of Remark \ref{rem_natural_dilation}. 
This is also the case for the properties in Theorem \ref{thm_christ}.  

\medskip 

Our second observation is that the two ways of determining the constant 
directly from the Plancherel formula in Theorem \ref{thm_christ} 
or from \eqref{eq_c0}  coincide for instance, if we change the positive Rockland operator $\cR$ for a multiple $c\cR$ for some $c>0$. Indeed, we check easily using  
\eqref{eq_c0} and \eqref{eq_homogeneitypsiR} or Theorem \ref{thm_christ}  and a change of variable on $(0,\infty)$
$$
c_0(c\cR) = c^{-Q/\nu} c_0(\cR),
$$

Let us also check what happens if we compare the formulae for 
a positive Rockland operator  $\cR$  and  any of its positive powers $\cR^\ell$ - which is also a positive Rockland operator. 
A simple change of variable on $(0,+\infty)$ implies
\begin{equation}
\label{eq_c0Rl}
c_0(\cR^\ell) = \frac 1 \ell c_0(\cR)
\quad \mbox{for any}\ \ell\in \bN.
\end{equation}
This together with \eqref{eq_c0} imply the following properties for the heat kernels $p_{\cR,t}$ and $p_{\cR^\ell,t}$ of $\cR$ and $\cR^\ell$:
\begin{equation}
\label{eq_p10Rl}
p_{\cR,1}(0) = p_{\cR^\ell,1}(0)
 \frac{\ell \Gamma(Q/\nu)}{\Gamma(Q/(\ell\nu))}.
 \end{equation}
 Let us sketch an independent proof for this last formula. 
 It will use the representation of $e^{-\lambda^\alpha}$, $\alpha\in (0,1)$, as a Laplace integral  \cite{pollard46}:
 \begin{lemma}
\label{lem_pollard}
 Let $\alpha\in (0,1)$. We have for any $\lambda\geq 0$
 $$
 e^{-\lambda^\alpha} = \int_0^\infty e^{-\lambda s}\phi_\alpha(s) ds,
$$
 where  $\phi_\alpha$ is a known  non-negative measurable function  on $(0,\infty)$ of integral one. 
 \end{lemma}
 
 \begin{remark}
 In this paper, we will not use the expression for $\phi_\alpha$:
 $$
 \phi_\alpha(s)=\frac 1\pi \int_0^\infty e^{-tu} e^{-u^\alpha \cos \pi \alpha}\sin (u^\alpha \sin \pi\alpha) ds.
 $$
 \end{remark}

 Using Lemma \ref{lem_pollard},  properties from functional calculus and the homogeneity of the heat kernel from \eqref{eq_homogeneitypsiR} yield
$$
p_{\cR,1}(0) = p_{\cR^\ell,1}(0) c'_{\alpha,\beta}
\quad\mbox{where}\quad
c'_{\alpha,\beta}:=\int_0^\infty s^{-\alpha \beta} \phi_\alpha (s) ds,
\quad \alpha:=\frac 1\ell,\ \beta:=  \frac Q \nu.
 $$
We also compute $c'_{\alpha,\beta}=  \frac{ \Gamma(\beta)}{\alpha\Gamma(\alpha\beta)}$
 using \eqref{eq_lambda-sGamma} and Lemma \ref{lem_pollard}.

 \medskip
 
 Our third observation is that it is possible to check Theorem \ref{thm_christ} in familiar situations. 
 For instance, if the positive Rockland operator is the canonical Laplacian $\Delta_{\bR^n} = -\sum_j \partial_j^2$ on the abelian group $G=\bR^n$, 
using the Euclidean Fourier transform and a change in polar coordinates, we readily obtain 
$c_0 (\Delta_{\bR^n})=  (\Gamma(n/2) 2^n \pi^{n/2})^{-1}$
which is indeed equal to $p_1(0)/\Gamma(n/2)$ since  the heat kernels in this setting are the Gaussians $p_t (x) = (4\pi t)^{-n/2} e^{-|x|^2/4t}$.

 We can also check Theorem \ref{thm_christ} and the formula in \eqref{eq_c0} in the case of the Heisenberg group.
Realising the Heisenberg group $\bH_n$ as $\bR^{2n+1}=\bR^n\times\bR^n\times \bR$ with the group law
$$
(x,y,t)(x',y',t') = (x+x',y+y',t+t'+\frac 12 (x\cdot y' - y\cdot x')), 
$$
the canonical sub-Laplacian $\cL_{\bH_n}$ is well understood, see e.g. \cite{Than98}  or \cite{R+F_monograph}. In particular, the relations between the  Fourier transform of $\bH_n$ and the functional calculus of $\cL_{\bH_n}$ yield
\begin{align*}
\|\psi(\cL_{\bH_n})\delta_0\|_{L^2(\bH_n)}^2 
&= (2\pi)^{-(3n+1)}\int_{-\infty}^{+\infty} \sum_{a=0}^\infty \frac {(n+a-1)!}{(n-1)! a!} |\psi(|\lambda|(2a+n))|^2 |\lambda|^n d\lambda
\\&=c_0 (\cL_{\bH_n})\int_0^\infty |\psi(\lambda)|^2 \lambda^{Q/2} \frac{d\lambda}\lambda,
\end{align*}
with $Q=2n+2$, so
$$
c_0(\cL_{\bH_n})= 
 (2\pi)^{-(3n+1)}2 \sum_{a=0}^\infty \frac {(n+a-1)!}{(n-1)! a!} (2a+n)^{-n-1}.
 $$
 The constant $c_0(\cL_{\bH_n})$ can also be determined via \eqref{eq_c0} and the known expression of the heat kernel due to Gaveau \cite{gaveau77}. Naturally, 
 the known expressions of the heat kernel via representation theory would lead to the same expression as above. 
 Using our method, we can also determine the constant $c_0$ for the positive Rockland operator $\cR= \cL_{\bH_n}^2 - \partial_t^2$.

\subsection{Properties of positive Rockland operators on $M$}
\label{subsec_propRM}

This section is devoted to the general properties of positive Rockland operators.
Many of these properties, for instance regarding self-adjointness and heat kernels,
are well-known for general sub-Laplacians on smooth manifolds \cite[p. 261-262]{Strichartz86}. Some of them are already known for Rockland operators on the compact filtered manifolds from the recent paper \cite{Dave+Haller}.

\smallskip

Let $\cR$ be a positive Rockland operator on $G$.
The operator $\cR_M$  it induces on $M$ is a smooth differential operator which is positive and essentially self-adjoint on $L^2(M)$, see  Section \ref{subsec_OpGM}. 
We will keep the same notation for $\cR_M$  and for its self-adjoint extension.
The properties of the functional calculus for $\cR$ imply:

\begin{lemma}
\label{lem_psiRM}
 Let $\cR$ be a positive Rockland operator on $G$.
 Let $\psi\in \cS(\bR)$.
 
 	The operator $\psi(\cR_M)$ defined as a bounded spectral multiplier on $L^2(M)$ coincides with the operator 
$$
\phi \longmapsto (\psi (\cR) \phi_G	)_M= (\phi_G * \kappa
)_M
$$
where  $\kappa:=\psi( \cR)\delta_0\in \cS(G)$.
 
 The integral kernel of $\psi(\cR_M)$ is a smooth function on $M\times M$ given by 
  $$
K(\dot x,\dot y) = 
\sum_{\gamma\in \Gamma}  \kappa(y^{-1}\gamma  x).
$$
The operator $\psi(\cR_M)$ is Hilbert-Schmidt on $L^2(M)$.
 \end{lemma}

\begin{proof}
This follows from 
Hulanicki's theorem (Theorem \ref{thm_hula})
and Lemma \ref{lem_IntKernel}.	
\end{proof}

The following statement is classical for sub-Laplacians on compact manifolds. We provide here a self-contained proof for $\cR_M$:

\begin{proposition}
\label{prop_sp}
\begin{enumerate}
\item The spectrum ${\rm Sp}(\cR_M)$ of 	$\cR_M$ is a discrete and unbounded subset of $[0,+\infty)$.
Each eigenspace of $\cR_M$ has finite dimension.
\item The constant functions on $M$ form the 0-eigenspace of $\cR_M$.
\end{enumerate}
\end{proposition}

Note that a consequence of Part (1) is that the resolvent operators $(\cR_M -z)^{-1}$, $z\in \bC \setminus {\rm Sp}(\cR_M)$,  are  compact on $L^2(M)$.

\begin{proof}[Proof of Proposition \ref{prop_sp} (1)]
The heat operator $e^{-t\cR_M}$ is positive and Hilbert-Schmidt  on $L^2(M)$, 
so its spectrum is a discrete subset of $[0,\infty)$ and the only possible accumulation point is zero.
Furthermore, the eigenspaces of $e^{-t\cR_M}$ for positive eigenvalues have finite dimensions. 
Let us show that the kernel of each $e^{-t\cR_M}$ is trivial. 
If $e^{-t\cR_M}f=0$ for some $f\in L^2(M)$ then 
$e^{-t'\cR_M}f=0$ for $t'=t, t/2,t/4,\ldots$
since $\|e^{-(t'/2)\cR_M}f\|_{L^2(M)}^2=(e^{-t'\cR_M}f,f)=0$.
By Section \ref{subsec_periodicfcn}, 
$f_G$ is a periodic tempered distribution in $L^2_{loc}(G)$ satisfying  $e^{-t' \cR} f_G=0$ 
for $t'=t, t/2,t/4,\ldots$, but this implies $f_G=0$ 
since 
$e^{-s \cR}$ converges to the identity on $\cS'(G)$ as $s\to 0$ by Lemma \ref{lem_appId}. So $f=0$. We have obtained that the kernel of each $e^{-t\cR_M}$ is trivial and thus that their spectrum is included in $(0,+\infty)$.

The spectrum of $\cR_M$ is the discrete set 
  ${\rm Sp}(\cR_M) = -\ln {\rm Sp} (e^{-\cR_M}) \subset \bR$.
It is unbounded since  $\cR_M$ is a (non-constant) differential operator. It is included in $[0,+\infty)$ as $\cR_M$ is positive. 
The eigenspaces for $\cR_M$ and for its heat kernels are in one-to-one correspondence, and therefore have finite dimensions. 
\end{proof}

\begin{proof}[Proof of Proposition \ref{prop_sp} (2)]
If a function is constant on $M$, then it is a 0-eigenfunction. Let us prove the converse. 
Let $f$ be a 0-eigenfunction, i.e.  $f\in L^2(M)$ and $\cR_M f=0$.
By Section \ref{subsec_periodicfcn}, 
$f_G$ is a periodic tempered distribution in $L^2_{loc}(G)$ satisfying $\cR f_G=0$.
By the Liouville theorem for homogeneous groups \cite[Theorem 3.2.45]{R+F_monograph} due to Geller \cite{Geller83}, 
 $f_G$ is a polynomial on $G\sim \bR^n$.
 As it is periodic, it must be a constant. Hence $f$ is a constant. 
\end{proof}

A consequence of Sections \ref{subsec_OpGM} and \ref{subsubsec_posRdef} is that the operator $\cR_M$ is hypoelliptic on the manifold $M$.
The same argument shows that  the operator $\cR_M -E {\rm I}$  is hypoelliptic for every constant $E\in \bC$, and this implies:
\begin{lemma}
 The eigenfunctions of $\cR_M$ are smooth on $M$.
\end{lemma}

\section{Asymptotics}
\label{sec_asymptotics}

In this section, we consider a nilmanifold $M=\Gamma\backslash G$ which is the quotient of a graded Lie group $G$ by a co-compact discrete subgroup $\Gamma$. 
We start with  obtaining asymptotics for traces of operators on $M$ coming from convolution operators on $G$.
We  apply these results to spectral multipliers in a positive Rockland operator, and this will  show  Theorem \ref{thm_main} (i).

\subsection{Asymptotics}
\label{subsec_asymptoticsHS}

The main technical result of this section is the following:

\begin{proposition}
\label{prop_trace}
For every $\kappa\in \cS(G)$,  the  operator defined via 
$$
T_\kappa (\phi) = (\phi_G*\kappa)_M, \qquad \phi\in L^2(M),
$$
is trace-class on $L^2(M)$. 
Denoting by  $\kappa^{(\eps)} \in \cS(G)$ for $\eps\in (0,1]$ the function given by
$$
\kappa^{(\eps)} (x):= \eps^{-Q}\kappa(\eps^{-1} x), \qquad x\in G, 
$$
the integral kernel  $K^{(\eps)}$ of  $T_{\kappa^{(\eps)}}$ is smooth and satisfies for $\eps$ small: 
$$ 
K^{(\eps)} (\dot x,\dot x) = \eps^{-Q} \kappa(0) +O(\eps)^\infty,
$$
and 
$$
\tr \left(T_{\kappa^{(\eps)}}\right)
= \eps^{-Q}\vol (M)\kappa(0) 
\ + \ O(\eps)^\infty.
$$
\end{proposition}
This means that for every $N\in \bN$ there exists a constant $C=C_{N,\kappa,G,\Gamma}>0$ such that for every $\eps\in (0,1]$
$$
\left|\tr \left(T_{\kappa^{(\eps)}}\right)
- \eps^{-Q}\vol (M) \kappa(0)  
\right|
\leq C \eps^N,
\quad\mbox{and}\quad
\forall \dot x\in M \quad \left|K^{(\eps)} (\dot x,\dot x) - \eps^{-Q} \kappa(0) \right|
\leq C \eps^N.
$$

We observe that $T_{\kappa_1}T_{\kappa_2} = T_{\kappa_2*\kappa_1}$ for any $\kappa_1,\kappa_2\in \cS(G)$ and $T_\kappa^*= T_{\tilde \kappa}$,  where  $\tilde \kappa(x) = \bar \kappa(x^{-1})$. 
Hence, applying Proposition \ref{prop_trace} to $\kappa*\tilde \kappa$, we obtain 
$$
\|T_{\kappa^{(\eps)}}\|_{HS}^2
= \eps^{-Q} \vol (M) \|\kappa\|_{L^2(M)}^2 
\ + \ O(\eps)^\infty.
$$

\begin{proof}[Proof of Proposition \ref{prop_trace}]
Lemma \ref{lem_IntKernel} implies that $T_{\kappa^{(\eps)}}$ is trace-class, with smooth
integral kernel
$$ 
K^{(\eps)} (\dot x,\dot y) = 
\sum_{\gamma\in \Gamma}\kappa^{(\eps)}(y^{-1} \gamma x).
$$

We fix a point $\dot x\in M$ by fixing  $ x= \Theta(t)\in G$ for $t$ in the fundamental domain $R_0$ described in  Proposition \ref{prop_FundDom}.
We may write 
$$
K^{(\eps)} (\dot x,\dot x) =  \sum_{\gamma\in \Gamma} \rho_{\gamma,\eps}, 
\qquad\mbox{where}\qquad 
\rho_{\gamma,\eps}:=\kappa^{(\eps)}(x^{-1} \gamma x) .
$$
Note that for $\gamma= 0$
$$
\rho_{0,\eps} = \kappa^{(\eps)}(0) = \eps^{-Q} \kappa(0). 
$$

We fix a homogenous quasi-norm $\|\cdot\|$ on $G$ (see Section \ref{subsubsec_dilations}). 
In order to avoid introducing unnecessary constants, 
we may assume that it yields a distance on $G\sim \bR^n$.
By assumption on the kernel $\kappa(z)$, 
we have
$$
\forall N\in \bN\quad \exists C_N \quad 
\forall z\in G\quad 
|\kappa (z)|\leq C_N (1+\|z\|)^{-N}.
$$
Consequently, fixing $N\in \bN$, 
$$
|\rho_{\gamma,\eps}|\leq
C_N \eps^{-Q}
 (1+ \eps^{-1} \|x^{-1} \gamma x\|)^N. 
$$
We observe that
 the function $t \mapsto 
\| \Theta(t)^{-1} \gamma  \Theta(t)\|$ is continuous from $\bR^n$ to $[0,+\infty)$. 
If $\gamma\not=0$, it never vanishes by Proposition \ref{prop_FundDom} and the properties of the quasi-norms;
let $c_\gamma>0$ denote its infimum.  
We deduce:
$$
(1+\eps^{-1}\| \Theta(t)^{-1} \gamma  \Theta(t)\|)^{-N}
   \leq (1+c_\gamma \eps^{-1})^{-N}\leq c_\gamma^{-N} \eps^{N}.
$$
We will use this for the finite number of $\gamma \in \Gamma\setminus\{0\}$ such that 
$\|\gamma\| \leq  4\max_{t\in \bar R_0}\| \Theta(t)\|$.
For the others,  the triangle inequality and  $\|\gamma\| > 4\max_{t\in \bar R_0}\| \Theta(t)\|$ imply that $\| \Theta(t)^{-1} \gamma  \Theta(t)\|\geq  \|\gamma\|/2$, so
$$
\int_{R_0}(1+\eps^{-1}\| \Theta(t)^{-1} \gamma  \Theta(t)\|)^{-N}
  dt \leq (1+\eps^{-1} \|\gamma\|/2)^{-N}
  \leq 2^N  \eps^{N} \|\gamma\|^{-N} .
  $$

Summing over $\gamma\in\Gamma\setminus\{0\}$, we obtain the estimate
$$
\sum_{\gamma\in \Gamma\setminus\{0\}}|\rho_{\eps,\gamma}|
\leq 
\epsilon^N\left(  \sum_{0<\|\gamma\| \leq  4\max_{t\in \bar R_0}\| \Theta(t)\|} c_\gamma^{-N} 
+2^N  \sum_{\|\gamma\| >  0} \|\gamma\|^{-N}\right).
$$
By Lemma \ref{lem_sum}, the very last sum is finite for $N$ large, $N>n\upsilon_n$ being sufficient.
 Hence  the right-hand side  above is $\lesssim \eps^N$. 
 This yields the estimates for $K^{(\eps)}(\dot x,\dot x)$. 
 Taking the integral over $\dot x\in M$ shows the trace expansion and   concludes the proof of Proposition \ref{prop_trace}. 
\end{proof}

\subsection{Applications to $\cR_M$}

We now  consider a positive Rockland operator $\cR$ on  $G$ and its corresponding operator $\cR_M$ on  $M=\Gamma \backslash G$.

\subsubsection{General result}
Functional calculus and previous results imply:

\begin{proposition}
\label{prop_traceR}
For any  $\psi\in \cS(\bR)$, 
the operator $\psi(t\cR_M)$ is trace-class.
Its integral kernel $K_{\psi(t\cR_M)}(\dot x, \dot y)$ is smooth on $M\times M$ and satisfies 
$$
K_{\psi(t\cR_M)} (\dot x,\dot x)= t^{-Q/\nu} \kappa(0) +O(t)^\infty.
$$
Here $\kappa(0)$ is the value at $x=0$ of the convolution kernel $\kappa= \psi(t\cR)\delta_0$, which is Schwartz by Hulanicki's theorem (Theorem \ref{thm_hula}), 
and we have by Corollary \ref{cor_thm_christ}
$$
\kappa(0)=c_0\int_0^\infty \psi (\lambda)  \lambda^{\frac Q \nu} \frac{d\lambda}{\lambda}.
$$
Here $\nu$ denotes the homogeneous degree of $\cR$, $Q$ the homogeneous degree of $G$
and $c_0$ is the constant from Theorem \ref{thm_christ}.

Furthermore,
the function defined  via
$$
(0,\infty)\ni t \longmapsto 
\tr \left(\psi(t\cR_M)\right) 
$$
is smooth and we have
$$
\tr \left(\psi(t\cR_M)\right) 
 = t^{-Q/\nu} 
\vol (M)\kappa(0) + O(t)^\infty.
$$
The Hilbert-Schmidt norm satisfies:
$$
\|\psi(t\cR_M)\|^2_{HS} = t^{-Q/\nu} 
\vol (M)\, c_0\int_0^\infty |\psi (\lambda)|^2  \lambda^{\frac Q \nu} \frac{d\lambda}{\lambda} + O(t)^\infty.
$$
\end{proposition}

\begin{proof}
Functional calculus guarantees that $t \longmapsto 
\tr \left(\psi(t\cR_M)\right) $ is well defined and differentiable with first derivative
$t \longmapsto 
\tr \left(\cR_M \psi(t\cR_M)\right) $. The smoothness follows recursively. 
The asymptotics are  consequences of Proposition \ref{prop_trace}.
\end{proof}

Proposition \ref{prop_traceR} implies the first part in Theorem \ref{thm_main}.

By Proposition \ref{prop_sp},
we may order the eigenvalues of $\cR_M$ (counted with multiplicity) into the sequence
$$
0=\lambda_0 < \lambda_1 \leq \lambda_2 \leq  \ldots \leq \lambda_j   \longrightarrow  +\infty
\quad\mbox{as}\quad j\to +\infty.
$$
Keeping the setting of Proposition \ref{prop_traceR},
by functional calculus, we have
$$
\tr \left(\psi(t\cR_M)\right)=
 \sum_{j=0}^\infty \psi(t\lambda_j). 
$$

\subsubsection{Heat expansions}

Applying Lemma \ref{lem_psiRM} and Proposition \ref{prop_traceR} to $\psi(\lambda) = e^{-\lambda}$ yield:

\begin{proposition}
\label{prop_theta}
\begin{enumerate}
\item The heat operators $e^{-t \cR_M}$, $t>0$, admit the following smooth kernels $K_t$ on $M\times M$:
$$
K_t(\dot x,\dot y) = 
\sum_{\gamma\in \Gamma}  p_t(y^{-1}\gamma  x),
\quad \dot x,\dot y\in M,
$$
where $p_t=e^{-t\cR}\delta_0$, $t>0$, are the heat kernels for $\cR$.
It satisfies as $t$ goes to 0
$$ 
K_t (\dot x,\dot x) = t^{-Q/\nu} p_1(0) +O(t)^\infty. 
$$
\item The function $\theta = \theta_{\cR_M}$ defined  by the heat trace:
\begin{equation}
\label{eq_deftau}
\theta(t):= \tr \left(e^{-t\cR_M}\right), \qquad t>0,
\end{equation}
is positive valued, decreasing and smooth on $(0,\infty)$.
\item As $t\to 0$, it satisfies the asymptotics:
$$
\theta(t)=
\vol (M) p_1(0) t^{-Q/\nu}+O(t)^\infty,
$$
where $p_1(0)=e^{- \cR}(0)$ is the heat kernel at time $t=1$ and $x=0$. 

\item 
Fixing any $\gamma \in (0,\lambda_1)$ where $\lambda_1$ is the first non-zero eigenvalue of $\cR_M$, we have
$$
|\theta (t) - 1| 
\leq C_\gamma e^{-\gamma t},
$$
for any $t\geq 1$ with the positive finite constant $C_\gamma :=  \sum_{j=1}^\infty e^{-(\lambda_j-\gamma) }$.
\end{enumerate}
\end{proposition}

\begin{proof}
By Proposition \ref{prop_traceR} applied to $\psi(\lambda)=e^{-\lambda}$, 
the kernel $K_t$ and 
the function $\theta$ are smooth on $M\times M$ and $(0,\infty)$ respectively,  and admit  asymptotic of the form given in Points (1) and  (3).   It is positive because
$$
\theta(t)= \tr \left(e^{-t\cR_M}\right) = \|e^{-\frac t 2 \cR_M}\|_{HS}^2, \qquad t>0,
$$
and decreasing since we have for  $t>t_0$ 
$$
\theta(t) = \tr \left(e^{-t_0 \cR_M} e^{-(t-t_0)\cR_M}\right) 
\leq
\|e^{-(t-t_0)\cR_M}\|_{\sL(L^2(M))}  \tr \left(e^{-t_0 \cR_M} \right) 
\leq 
\theta(t_0).
$$
Point (4) follows from
$$
|\theta (t) - 1| = \theta(t)-1 = \sum_{j=1}^\infty e^{-\lambda_j t}
\leq e^{-\gamma t}  \sum_{j=1}^\infty e^{-(\lambda_j-\gamma) t}.
$$
This concludes the proof of Proposition \ref{prop_theta}. 
\end{proof}

We observe that Lemma \ref{lem_pollard} implies (with its notation):
$$
\theta_{\cR_M} (t)= \int_0^\infty \theta_{\cR_M^\ell} (t^\ell s) \, \phi_{1/\ell}(s) ds, 
\qquad \ell\in \bN, \ t>0. 
$$
We will not use this.

\subsubsection{Weyl law}

We denote the spectral counting function by
$$
N(\Lambda):=\left|\left\{ j\in\bN_0,\;\;\lambda_j\leq \Lambda\right\}\right|.
$$

Our previous analysis classically implies the Weyl law for $\cR_M$:

\begin{theorem}[Weyl law]\label{thm_Weyl}
We have
  $$ 
  \lim_{\Lambda \to +\infty} \Lambda^{-Q/\nu }
  N(\Lambda) = \vol (M)\, p_1(0)/ \Gamma(1+ \nu/ Q )
  $$
	where $Q$ is the homogeneous dimension of $G$ and $\nu$ the homogeneous degree of $\cR$.
\end{theorem}

Indeed, the Weyl law is a consequence of the heat kernel trace via Karamata's Tauberian theorem, see e.g. \cite[p.116]{Roe}. 

Another classical proof comes from taking  $\eps^\nu=\Lambda^{-1}$ and  a convenient choice of functions $\psi\in \cD(\bR)$ approximating the indicatrix of $[0,1]$  in Proposition \ref{prop_traceR};
the constant is simplified using \eqref{eq_c0}. 
This approach has the advantage that it can be modified to prove a slightly more general result.  Indeed, by taking approximate indicatrices of a segment $[a,b]$, we obtain the following generalised Weyl law for any  $0\leq a<b$:
the semi-classical counting function for $[a,b]$  admits the following asymptotic
$$
 \{j\in \bN_0 \ : \ \Lambda a \leq  \lambda_j\leq \Lambda b\} \sim 
  \frac{ \vol (M)\, p_1(0)}{ \Gamma(1+ \nu/ Q )} (b^{\frac Q\nu} -a^{\frac Q\nu}) \Lambda^{Q /2},
\quad\mbox{as}\ \Lambda\to +\infty. 
$$

\medskip

Let us make some comments on Theorem \ref{thm_Weyl}. 
\begin{remark}
\label{rem_thm_Weyl}
\begin{enumerate}
\item In the particular case of the canonical Laplacian $\Delta_{\bT^n} = -\sum_j \partial_j^2$ on the torus $\bT^n=\bR^n / \bZ^n$, we recover the well known result since 
$\nu=2$, $Q=n$, $\vol (M)=1$ and 
$c_0 (\Delta_{\bR^n})=  (\Gamma(n/2) 2^n \pi^{n/2})^{-1}$
(see Section \ref{subsubsec_c_0}).

\item Let us consider the case of the canonical Heisenberg nil-manifold, that is, the quotient $M=\Gamma \backslash\bH_n$ 
of the  Heisenberg group $\bH_n$ by the canonical lattice $\Gamma = \bZ^n \times \bZ^n \times \frac 12 \bZ$.
The spectrum  of the canonical sub-Laplacian $\cL_M$
 is  known \cite{Deninger+Singhof,Folland2004,thangavelu2009}:
 it consists of the single eigenvalue
 $4\pi^2 |m|^2$ where $m$ runs over $\bZ^{2n}$,
 and of the eigenvalue 
$4(2a+n)\pi |k|$
with multiplicity $(2|k|)^n\frac{(n+a-1)!}{(n-1)!a!}$
where $a$ and $k$ run over $\bN$ and $ \bZ\setminus\{0\}$ respectively.

Since $\vol(M)=1/2$, 
the Weyl law for $\cL_M$ gives  as $\Lambda\to +\infty$
$$
c_1 \Lambda^{n+1}\sim 
\sum_{m\in \bZ^{2n} :  
 4\pi^2 |m|^2 \leq \Lambda} 1 + \sum_{\substack{k\in \bZ\setminus \{0\},\, a\in \bN \\ 4(2a+n)\pi |k|<\Lambda} } (2|k|)^n\frac{(n+a-1)!}{(n-1)!a!}
 $$
 where $c_1 = c_0(\cL_{\bH_n}) / (2n+2)$, and 
 the constant $c_0(\cL_{\bH_n})$ was explicitly given in Section \ref{subsubsec_c_0}.
  The Weyl law for the torus implies that 
 the first sum is $\sim c'_1 \Lambda^{n}$.
 Hence we have obtained:
 $$
c_1 \Lambda^{n+1}\sim   
\sum_{\substack{k\in \bZ\setminus \{0\},\, a\in \bN \\ 4(2a+n)\pi |k|<\Lambda} } (2|k|)^n\frac{(n+a-1)!}{(n-1)!a!}.
 $$

  \item If $\cR$ is a positive Rockland operator, then any positive powers of $\cR$ is also a positive Rockland operator
and we can check using the property \eqref{eq_p10Rl} of the constant $p_1(0)$ that the Weyl law above for $\cR$ implies the Weyl law for $\cR^\ell$ for any $\ell\in \bN$. 

We can also check that the Weyl law for $\cR$ is equivalent to the Weyl law for any positive multiple of $\cR$ and that the quotient $\Lambda^{-Q/\nu}N(\Lambda)$ is independent of this multiple. 

\item\label{item_rem_thmWeyl_Q/nu} The ratio $Q/\nu$ is independent of a choice of dilations adapted to the gradation of $G$ in the sense of Remark \ref{rem_natural_dilation}.

\item\label{item_rem_thmWeyl_lambdaj} From $j = N(\lambda_j) \sim c' \lambda_j^{\nu/Q}$ with $c'=\vol (M)\, p_1(0)/ \Gamma(1+ \nu/ Q )$, we deduce 
$$
\lambda_j \sim \left(\frac j{c'}\right)^{Q/\nu}
\quad\mbox{as}\ j\to \infty. 
$$
\end{enumerate}
\end{remark}

\section{Zeta functions for $\cR_M$}

As before, we consider a nilmanifold $M=\Gamma\backslash G$ which is the quotient of a graded Lie group $G$ by a co-compact discrete subgroup $\Gamma$. 
We also consider a positive Rockland operator $\cR$ on $G$ and the corresponding operator $\cR_M$ on $M$. 
The eigenvalues of $\cR_M$ (counted with multiplicity) are ordered  into the sequence
$$
0=\lambda_0 < \lambda_1 \leq \lambda_2 \leq  \ldots \leq \lambda_j   \longrightarrow  +\infty
\quad\mbox{as}\quad j\to +\infty.
$$

In this section, we are interested in the zeta function for $\cR_M$ 
which is defined formally by
\begin{equation}
\label{eq_def_zeta}
\zeta_{\cR_M}(s) =\sum_{j=1}^\infty \lambda_j^{-s}.
\end{equation}
More precisely, 
we will show the second part of Theorem \ref{thm_main}.

\subsection{Meromorphic extension}
{It is a standard consequence of  the properties of the theta function (see
Proposition \ref{prop_theta}) and of the Mellin transform that $\zeta_{\cR_M}$ will have a meromorphic extension whose pole can be computed}:

\begin{theorem}
\label{thm_zeta}
The sum in \eqref{eq_def_zeta} is absolutely convergent for $\Re s>\nu /Q$
where $\nu$ is the homogeneous degree of $\cR$ and $Q$ the homogeneous dimension of $G$. 
Hence, this defines the holomorphic function  $\zeta_{\cR_M}$ on $\{\Re s>\nu /Q\}$. We can write
$$
\zeta_{\cR_M}(s)  = \frac1{\Gamma(s)} \frac {\vol (M) p_1(0)}{s- Q/\nu} +h(s),
$$
where $h$ is an entire function.
Consequently, 
$\zeta_{\cR_M}$ admits a meromorphic extension to $\bC$, with only one  pole.
The pole is simple, located at $s=Q/\nu$ and with  residue  $  \frac {\vol (M) p_1(0)}{\Gamma(Q/\nu)}$.
\end{theorem}

One checks easily that for $\ell\in \bN$ and $c>0$, 
$$
\zeta_{c\cR_M} (s) =c^{-s} \zeta_{\cR_M}(s) 
\quad\mbox{and}\quad
\zeta_{\cR_M^\ell}(s) = \zeta_{\cR_M}(\ell s),
$$
and that the relations between the simple poles and residues of the various zeta functions following from these relations and also from Theorem \ref{thm_zeta} are consistent. 
Moreover, the properties in Theorem \ref{thm_zeta} are independent of a choice of dilations adapted to the gradation of $G$ in the sense of Remark \ref{rem_natural_dilation}.

\begin{proof}
 Remark \ref{rem_thm_Weyl} \eqref{item_rem_thmWeyl_lambdaj}
implies the absolute convergence from which the holomorphy follows. 
Denoting the heat trace by $\theta$ as in \eqref{eq_deftau} and using the Mellin transform and \eqref{eq_lambda-sGamma},
we can write at least formally 
\begin{align*}
\zeta_{\cR_M}(s) 
&= \frac1{\Gamma(s)} \int_0^\infty t^{s-1}\sum_{j=1}^\infty  e^{-t\lambda_j} dt
= \frac1{\Gamma(s)} \int_0^\infty t^{s-1} (\theta(t) -1) dt
\\&= 
h_1(s)+
\frac1{\Gamma(s)} \int_0^1 t^{s-1} \vol (M) p_1(0) t^{-Q/\nu}  dt
 - \frac1{\Gamma(s)} \int_0^1 t^{s-1} dt+ h_2(s) 
\\&=  h_1(s)+ \frac1{\Gamma(s)} \frac {\vol (M) p_1(0)}{s- Q/\nu} - \frac1{\Gamma(s+1)} + h_2(s),
\end{align*}
where
$$
h_1(s):=
 \frac1{\Gamma(s)} \int_1^\infty t^{s-1} (\theta(t) -1) dt, 
\quad\mbox{and}\quad
h_2(s):=
\frac1{\Gamma(s)} \int_0^1 t^{s-1} (\theta(t) -\vol (M) p_1(0) t^{-Q/\nu} ) dt.
$$
The exponential bound of $\theta(t)$ for $t\geq 1$ 
and its asymptotics as $t\to 0$ in Proposition \ref{prop_theta} 
imply that $h_1$ and $h_2$ are entire, and  that the decomposition of $\zeta_{\cR_M}$ above holds. 
This concludes the proof. 
\end{proof}

\subsection{Further properties}

\begin{lemma}
\label{lem_cR1+cR2}
Assume that $\cR_1$ and $\cR_2$ are two positive Rockland operators on $G$ with the same degree of homogeneity. 
Then $\cR_1+\cR_2$ is also a positive Rockland operator in $G$. 
If $\cR_{1,M}$ and $\cR_{2,M}$ commute strongly (i.e. their resolvents commute), then 
$$
\zeta_{\cR_{1,M}+\cR_{2,M}}(s) = 
\zeta_{\cR_{1,M}} (s)+
\zeta_{\cR_{2,M}} (s)+
Z(s) ,
$$
where 
$$
Z(s) = 
\frac 1{\Gamma(s)}
\int_0^\infty 
\left(\theta_{\cR_{1,M}}(t)-1\right)
\left(\theta_{\cR_{2,M}}(t)-1\right)
t^{s-1} dt.
$$
\end{lemma}

\begin{proof}
This follows directly from the computation:
$$
\zeta_{\cR_{1,M}+\cR_{2,M}}(s) = 
\sum_{j_1+j_2>0}
(\lambda_{1,j_1}+\lambda_{2,j_2})^{-s}
=
\zeta_{\cR_{1,M}} (s)+
\zeta_{\cR_{2,M}} (s)+
Z(s) ,
$$
where 
$$
Z(s) = \sum_{j_1,j_2>0}
(\lambda_{1,j_1}+\lambda_{2,j_2})^{-s}
=
\frac1{\Gamma(s)} \int_0^\infty t^{s-1}
 \sum_{j_1,j_2>0}  e^{-t(\lambda_{1,j_1}+\lambda_{2,j_2})} dt.
$$
\end{proof}

An example of a setting where we can apply Lemma \ref{lem_cR1+cR2} is the following. 
Consider two graded nilpotent Lie groups $G_1$, $G_2$, equipped respectively with co-compact discrete subgroups $\Gamma_1$ and $\Gamma_2$, and positive Rockland operators $\cR_1$ and $\cR_2$ with the same degree. 
We consider the associated nilmanifolds $M_1$ and $M_2$, 
and operators $\cR_{1,M_1}$ and $\cR_{2,M_2}$.
Then the direct product $G=G_1\times G_2$ is a graded nilpotent Lie group, 
equipped with the co-compact discrete subgroup $\Gamma_1\times \Gamma_2$. 
The operators $\cR_1\otimes \id$ and  $\id \otimes \cR_2$ are two positive Rockland operators on $G$ with the same degree and their associated operators
$\cR_{1,M_1}\otimes \id$ and  $\id \otimes \cR_{2,M_2}$ on $M=M_1\times M_2 =\Gamma \backslash G$ commute strongly.

\medskip

{Following the ideas of  \cite{Bauer+Furutani2010}, 
this example above will allow us to calculate the `trivial zeros' of $\zeta_{\cR_M}$ and its value at $s=0$.} 
Indeed, a particular case of this  setting  is obtained by considering $G_1$ a given stratified Lie group equipped with a sub-Laplacian operator $\cR_1=\cL$ (see \eqref{eq_cL})
and $G_2=\bR$ with $\cR_2= \Delta_{\bR}$. 
With $\Gamma_2=\bZ$ and $\cR_{M_2} = \Delta_{\bT}$, we compute:
$$
\zeta_{\cR_2,\bT}=
\zeta_{\Delta_{\bT}}(s)= 2 (2\pi)^{-s} \zeta(2s),
$$
where $\zeta(s) = \sum_{j\geq 1} j^{-s}$ is the Riemann zeta function, 
and by the Poisson summation formula:
$$
\theta_{\cR_{2,M_2}}(t)=
\theta_{\Delta_\bT}(t)= \frac1 {2\sqrt{\pi t}} \sum_{j\in \bZ} e^{-\frac{j^2}{4t}}.
$$
This implies readily that the function $Z(s)$ from Lemma \ref{lem_cR1+cR2} in this setting satisfies
\begin{align*}
Z(s)
&=
-\zeta_{\cR_{1,M_1}} (s)
+\frac {\Gamma(s-\frac12)}{\sqrt \pi \Gamma(s)}\zeta_{\cR_{1,M_1}} (s-\frac 12)
+h(s)
\end{align*}
where $h(s)$ is the entire function given by 
$$
h(s) := \frac 1{\sqrt \pi \Gamma(s)}
\int_0^\infty 
\left(\theta_{\cR_{1,M_1}}(t)-1\right)
\sum_{j=1}^\infty e^{-\frac{j^2}{4t}}
t^{s-\frac 32} dt.
$$
We have obtained
\begin{equation}
\label{eq_zetaR1+T}
\zeta_{\cR_{1,M_1}+\Delta_{\bT}} (s)= 
2 (2\pi)^{-s} \zeta(2s)+
\frac {\Gamma(s-\frac12)}{\sqrt \pi \Gamma(s)}\zeta_{\cR_{1,M_1}} (s-\frac 12)
+h(s).
\end{equation}
with $\cR_1=\cL$. 
As in \cite[Section 5]{Bauer+}, 
this formula  together with Theorem \ref{thm_zeta} and the properties of the Gamma and Riemann zeta functions imply  the following property for $\cL$ as Rockland operator :
\begin{proposition}
\label{prop_values}
We have
$$
\zeta_{\cR_M}(0)=-1
\quad\mbox{and}\quad
\zeta_{\cR_M}(s)=0 \ \mbox{for}\ s=-1,-2,\ldots
$$
\end{proposition}

This concludes the proof of Theorem \ref{thm_main} for sub-Laplacians. 
We end this paper with modifying the arguments above to show Proposition \ref{prop_values} for any positive Rockland operator. Let us consider 
a positive Rockland operator $\cR_1$  on a graded group $G_1$. 
We may  assume that its homogeneous degree is $\nu \in 2\bN$ even (see Remark \ref{rem_natural_dilation}). 
As above, we consider the group $G_2=\bR$ and the operator $\cR_2= \Delta_{\bR}$.
However,  
we equip $G_2=\bR$ with the dilations $r\cdot x =  r^{\nu/2} x$ so that the homogeneity of $\cR_2$ is now $\nu$.  
As above we obtain \eqref{eq_zetaR1+T} and conclude in the same way. 
This proves Proposition \ref{prop_values} for any positive Rockland operator $\cR_1$
and concludes completely the proof of Theorem \ref{thm_main}. 

\medskip

\section*{Acknowledgement}

This work is supported by the Leverhulme Trust, Research Project Grant 2020-037.

\end{document}